\documentclass[letterpage, 11pt, notitlepage]{article}

\usepackage[margin=1.0in]{geometry}

\usepackage{enumitem}
\usepackage{times}
\usepackage{graphicx} 
\usepackage{subfigure}

\usepackage{algorithm}
\usepackage{algorithmic}

\usepackage{hyperref}
\usepackage{breakurl}

\AtBeginShipout{%
  \ifnum\value{page}>1 %
    \typeout{* Additional boxing of page `\thepage'}%
    \setbox\AtBeginShipoutBox=\hbox{\copy\AtBeginShipoutBox}%
  \fi
}

\usepackage{pifont}

\usepackage{epsfig}
\usepackage{amssymb}
\usepackage{amsmath}
\usepackage{amsthm}
\usepackage{amsfonts}
\usepackage{bbding}
\usepackage{array}
\usepackage{caption,tabularx,booktabs}
\usepackage{arydshln}
\usepackage{paralist}
\usepackage{xargs}                      
\usepackage[pdftex,dvipsnames]{xcolor}  

\usepackage{multirow}

\newcommandx{\unsure}[2][1=]{\todo[inline,linecolor=red,backgroundcolor=red!25,bordercolor=red,#1]{#2}}
\newcommandx{\change}[2][1=]{\todo[linecolor=blue,backgroundcolor=blue!25,bordercolor=blue,#1]{#2}}
\newcommandx{\info}[2][1=]{\todo[linecolor=OliveGreen,inline,backgroundcolor=OliveGreen!25,bordercolor=OliveGreen,#1]{#2}}
\newcommandx{\improvement}[2][1=]{\todo[linecolor=Plum,inline,backgroundcolor=Plum!25,bordercolor=Plum,#1]{#2}}

\usepackage[font=small,skip=5pt]{caption}

\newtheorem{thm}{Theorem}
\newtheorem{lem}{Lemma}

\newtheorem{cor}{Corollary}

\newtheorem{ass}{Assumption}
\newtheorem{remark}{Remark}

\newcolumntype{C}[1]{>{\centering\let\newline\\\arraybackslash\hspace{0pt}}m{#1}}

\newcommand\tagthis{\addtocounter{equation}{1}\tag{\theequation}}

\DeclareMathOperator{\R}{\mathbb{R}} 
\DeclareMathOperator{\Ocal}{\mathcal{O}}

\makeatletter
\newcounter{subthm} 
\let\savedc@thm\c@hyp

\newcommand{\normhyp}{%
  \let\c@hyp\savedc@hyp 
  \renewcommand\thehyp{\arabic{hyp}}%
} 
\makeatother

\makeatletter
\newcounter{subass} 
\let\savedc@ass\c@hyp

\makeatother

\newcommand{\Eset}{\mathbb{E}}

\newcommand{\Pset}{\mathbb{P}}

\newcommand{\Rset}{\mathbb{R}}

\newcommand{\Acal}{{\cal A}}

\newcommand{\Fcal}{{\cal F}}

\newcommand{\Scal}{{\cal S}}

\newcommand{\Xcal}{{\cal X}}

\newcommand{\xbar}{{\bar{x}}}

\makeatother

\usepackage{blindtext,titlefoot}

\begin{document}

\title{\fontsize{20}{20}\selectfont Convergence Rates of Accelerated Markov Gradient Descent with Applications in Reinforcement Learning}

\author{
Thinh T. Doan
\and
Lam M. Nguyen
\and
Nhan H. Pham
\and
Justin Romberg
}

\maketitle
\date{}
\unmarkedfntext{\textbf{Thinh T. Doan}, School of Electrical and Computer Engineering, Georgia Institute of Technology, GA, USA. Email: \href{mailto:thinhdoan@gatech.edu}{thinhdoan@gatech.edu}}
\unmarkedfntext{\textbf{Lam M. Nguyen}, IBM Research, Thomas J. Watson Research Center, Yorktown Heights, NY, USA. Email: \href{mailto:lamnguyen.mltd@ibm.com}{LamNguyen.MLTD@ibm.com}}
\unmarkedfntext{\textbf{Nhan H. Pham}, Department of Statistics and Operations Research, University of North Carolina at Chapel Hill, Chapel Hill, NC, USA. Email: \href{mailto:nhanph@live.unc.edu}{nhanph@live.unc.edu}}
\unmarkedfntext{\textbf{Justin Romberg}, School of Electrical and Computer Engineering, Georgia Institute of Technology, GA, USA. Email: \href{mailto:jrom@ece.gatech.edu}{jrom@ece.gatech.edu}}

\begin{abstract}
Motivated by broad applications in machine learning, we study the popular accelerated stochastic gradient descent (ASGD) algorithm for solving (possibly nonconvex) optimization problems. 
We characterize the finite-time performance of this method when the gradients are sampled from Markov processes, and hence biased and dependent from time step to time step; in contrast, the analysis in existing work relies heavily on the stochastic gradients being independent and sometimes unbiased. Our main contributions show that under certain (standard) assumptions on the underlying Markov chain generating the gradients, ASGD converges at the nearly the same rate with Markovian gradient samples as with independent gradient samples.  
The only difference is a logarithmic factor that accounts for the mixing time of the Markov chain.
One of the key motivations for this study are complicated control problems that can be modeled by a Markov decision process and solved using reinforcement learning.  We apply the accelerated method to several challenging problems in the OpenAI Gym and Mujoco, and show that acceleration can significantly improve the performance of the classic temporal difference learning and REINFORCE algorithms.
\end{abstract}

\section{Introduction}
\label{sec:intro}
Stochastic gradient descent ({\sf SGD}) and its variants, originally introduced in \cite{RM1951} under the name of stochastic approximation ({\sf SA}), is the most efficient and widely used method for solving optimization problems in machine learning ({\sf RL}) \cite{bottou2016optimization,SVRG,SAGA,nguyen2017sarah} and reinforcement learning \cite{schulman2015trust,schulman2017proximal}. It can substantially reduce the cost of computing a step direction in supervised learning, and offers a framework for systematically handling uncertainty in reinforcement learning. In this context, we want to optimize an (unknown) objective function $f$ when queries for the gradient are noisy.  At a point $x$, we observe a random vector $G(x, \xi)$ whose mean is the (sub)gradient of $f$ at $x$. {\sf SGD} updates the iterate $x$ by moving along the opposite direction of $G(x,\xi)$ scaled by some step sizes. Through judicious choice of step sizes, the ``noise'' induced by this randomness is averaged out across iterations, and the algorithm converges to the stationary point of $f$; see for example \cite{bottou2016optimization} and the references therein.

To further improve the performance of {\sf SGD}, stochastic versions of Nesterov’s acceleration scheme  \cite{Nesterov1983AMF} have been studied in different settings  \cite{jain2017accelerating,vaswani2018fast,liu2018accelerating}. In many of these cases, it has been observed that acceleration improves the performance of {\sf SGD}  both in theory \cite{ GhadimiL2013, sutskever2013importance,dieuleveut2017harder} and in practice \cite{kasai2017sgdlibrary}, with a notable application in neural networks \cite{tensorflow2015-whitepaper}. This benefit of accelerated {\sf SGD} ({\sf ASGD}) has been studied under the i.i.d noise settings. Almost nothing is known when the noise is Markovian, which is often considered in the context of reinforcement learning ({\sf RL}) problems modeled by Markov decision processes \cite{Sutton2018}. 

In this paper, we show that a particular version of {\sf ASGD} is still applicable when the gradients of the objective are sampled from Markov process, and hence are biased and not independent across iterations. This model for the gradients has been considered previously in \cite{DuchiAJJ2012, SunSY2018,JohanssonRJ2010,RamNV2009}, where different variants of {\sf SGD} are considered. 
It has also been observed that the {\sf SGD} performs better, i.e., using less data and computation, when the gradients are sampled from Markov processes as compared to i.i.d samples in both convex and nonconvex problems \cite{SunSY2018}. 
This paper shows that the benefits of acceleration extend to the Markovian setting in theory and in numerical experiments; we provide theoretical convergence rates that nearly match those in the i.i.d.\ setting, and show empirically that the algorithm is able to learn from significantly fewer samples on benchmark reinforcement learning problems. Our goal is to draw a connection between stochastic optimization and {\sf RL}, in particular, when can we apply acceleration techniques to improve the performance of {\sf RL} algorithms?




\textbf{Main contributions}. 
We study accelerated stochastic gradient descent where the gradients are sampled from a Markov process, which we refer to as accelerated Markov gradient descent ({\sf AMGD}).  We show that, despite the gradients being biased and dependent across iterations, the convergence rate across many different types of objective functions (convex and smooth, strongly convex, nonconvex and smooth) is within a logarithmic factor of the comparable bounds in i.i.d settings.  This logarithmic factor is naturally related to the mixing time of the Markov process generating the stochastic gradients.  To our knowledge, these are the first such bounds for accelerated stochastic gradient descent with Markovian samples.

We also show that acceleration is extremely effective in experiments by applying it to multiple  problems in reinforcement learning.  Compared with the popular temporal difference learning and Monte-Carlo policy gradient REINFORCE algorithms, the accelerated variants require significantly fewer samples to learn a policy with comparable rewards, which aligns with our theoretical results.


\section{Accelerated Markov gradient descent}
We consider the (possibly nonconvex) optimization problem over a closed convex set $\Xcal\subset\Rset^{d}$
\begin{align}
\underset{x\in\Xcal}{\text{minimize}}\; f(x) ,\label{prob:main}
\end{align}
where  $f:\Xcal\rightarrow\Rset$ is given as
\begin{align}
f(x) \triangleq \Eset_{\pi}[F(x;\xi)] = \int_{\Xi}F(x;\xi)d\pi(\xi).\label{notation:f}
\end{align}
Here $\Xi$ is a statistical sample space with probability distribution $\pi$ and $F(\cdot;\xi):\Xcal\rightarrow\Rset$ is a bounded below (possibly nonconvex) function  associated with $\xi\in\Xi$. We are interested in the first-order stochastic optimization methods for solving problem \eqref{prob:main}. Most of existing algorithms, such as {\sf SGD}, require a sequence of $\{\xi_{k}\}$ sampled i.i.d from the distribution $\pi$. Our focus is to consider the case where $\{\xi_{k}\}$ are generated from an ergodic Markov process, whose stationary distribution is $\pi$. In many cases, using Markov samples are more efficient than i.i.d samples in implementing {\sf SGD} \cite{SunSY2018}.  


We focus on studying accelerated gradient methods for solving problem \eqref{prob:main}, originally proposed by Nesterov \cite{Nesterov1983AMF} and studied later in different variants; see for example \cite{Lan2012,GhadimiL2013,jain2017accelerating,vaswani2018fast} and the reference therein.  In particular, we show that the {\sf ASGD} algorithm proposed in \cite{Lan2012,Lan2019,GhadimiL2013} converges at (nearly) the same rate when the gradients are sampled from a Markov process as when they are sampled independently at every iteration.  Despite of its importance in applications, near-optimal convergence rates for {\sf ASGD} under Markovian noise have not yet appeared in the literature


In our algorithms,  $G(x;\xi)\in\partial F(x;\xi)$ is the (sub)gradient of $F(\cdot;\xi)$ evaluated at $x$. As mentioned we consider the case where  $\{\xi_{k}\}$ is drawn from a Markov ergodic stochastic process. We denote by $\tau(\gamma)$ the mixing time of the Markov chain $\{\xi_{k}\}$ given a positive constant $\gamma$, which basically tells us how long the Markov chain gets close to the stationary distribution \cite{LevinPeresWilmer2006}. To provide a finite-time analysis of this algorithm, we consider the following fairly standard assumption about the Markov process.

\begin{ass}\label{assump:ergodicity}
The Markov chain $\{\xi_{k}\}$ with finite state $\Xi$ is ergodic, i.e., irreducible and aperiodic. 

\end{ass}
Assumption \ref{assump:ergodicity} implies that $\{\xi_{k}\}$ has geometric mixing time\footnote{$\tau$ depends on the second largest eigenvalue of the transition probability matrix of the Markov chain.}  , i.e., given $\gamma>0$ there exists $C>0$ s.t. 
\begin{align}
 \tau(\gamma) &= C\log(1/\gamma) \quad\text{and} \quad \|\Pset^{k}(\xi_{0},\cdot) - \pi \|_{TV} \leq \gamma, \forall k\geq \tau(\gamma),\;\forall \xi_{0}\in\Xi,\label{notation:tau}    
\end{align}
where  $\|\cdot\|_{TV}$ is the total variance distance and $\Pset^{k}(\xi_{0},\xi)$ is the probability that $\xi_{k} = \xi$ when we start from $\xi_{0}$ \cite{LevinPeresWilmer2006}. This assumption holds in various applications, e.g, in incremental optimization \cite{RamNV2009}, where the iterates are updated based on a finite Markov chain. Similar observation holds in reinforcement learning problems that have a finite number of states and actions, for example in AlphaGo \cite{silver2017mastering}. Assumption \ref{assump:ergodicity} is used in the existing literature to study the finite-time performance of {\sf SA} under Markov randomness; see  \cite{SunSY2018,  SrikantY2019_FiniteTD, ChenZDMC2019, Doan2019, DoanMR2019_DTD} and the references therein.   

Before proceeding to the finite-time analysis of {\sf AMGD}, we present the motivation behind our approach and theoretical results given later. To study the asymptotic convergence of {\sf AMGD}, one may use the popular ordinary differential equation ({\sf ODE}) approach in stochastic approximation literature, see for example \cite{borkar2008,Kushner_Yin_book_2006}. On the other hand, our focus is to study the finite-time performance of {\sf AMGD}. The existing techniques in studying {\sf ASGD} rely on the main assumptions that the gradients are sampled i.i.d from the (unknown) stationary distribution $\pi$ and unbiased. In our setting, since the gradients are sampled from a Markov process, they are dependent and biased (\textit{nonstationary}). Even if we can sample from $\pi$ ($\tau= 0$ and the gradient samples are unbiased), they are still dependent. Thus, it is not trivial to handle the bias and dependence simultaneously using the existing techniques. We, therefore, utilize the geometric mixing time to  eliminate this issue in our analysis. Indeed, under Assumption \ref{assump:ergodicity}, we show that the convergence rates of the {\sf AMGD} are the same with the ones in {\sf ASGD} under i.i.d. samples for solving both convex and nonconvex problems, except for a $\log(k)$ factor which captures the mixing time $\tau$.

\section{Convergence analysis: Nonconvex case}\label{sec:nonconvex}
\begin{algorithm}[h]
\caption{{\sf AMGD} for nonconvex problems}
  {\bfseries Initialize:} Set arbitrarily $x_{k},\xbar_{k}$, step sizes $\{\alpha_{k},\beta_{k},\gamma_{k}\}$ for $k\leq 1$, and an integer $K\geq 1$\vspace{0.2cm}\\
  {\bfseries Iterations:} For $k = 1,\ldots,K$ do\vspace{-0.2cm}
    \begin{align}
    y_{k}  &= (1-\alpha_{k})\xbar_{k-1} + \alpha_{k}x_{k-1}\label{alg_nonconvex:y}\\
    x_{k} &= x_{k-1} - \gamma_{k} G(y_{k};\xi_{k})\label{alg_nonconvex:x}\\
    \xbar_{k} &= y_{k} - \beta_{k}G(y_{k};\xi_{k})\label{alg_nonconvex:xbar}
    \end{align}
   {\bfseries Output:} $y_{R}$ randomly selected from the sequence $\{y_{k}\}_{k=1}^{K}$ with probability $p_{k}$ defined as
\begin{align}
p_{k} = \frac{\gamma_{k}(1-L\gamma_{k})}{\sum_{k=1}^{K}\gamma_{k}(1-L\gamma_{k})}\cdot    
\label{thm_nonconvex:pk}
\end{align}  
\label{alg:ASGD_nonconvex}
\end{algorithm}
We study Algorithm \ref{alg:ASGD_nonconvex} for solving \eqref{prob:main} when $\Xcal = \Rset^{d}$, and $f$ is nonconvex satisfying the assumptions below. 
\begin{ass}\label{assump:lower_bound}
$f^{*} = \inf_{x\in\R^d} f(x) > -\infty$.
\end{ass}
In addition, we assume that $\nabla f$ and its samples are Lipschitz continuous and bounded, similar to the work in \cite{SunSY2018}.    
\begin{ass}\label{ass:Lipschitz}
There exists a constant $L > 0$ such that $\forall x,y$ and $\forall \xi\in\Xi$
\begin{align*}
& \| \nabla f ( x ) - \nabla f ( y ) \| \leq L \| x - y \|\qquad \text{and} \qquad \| G ( x;\xi ) - G ( y; \xi ) \| \leq L \| x - y \|. \tagthis \label{ass_Lipschitz:Ineq1}   
\end{align*}
\end{ass}
\begin{ass}\label{assump:bounded_gradient}
There exists a constant $M>0$ such that $\forall x$ and $\forall \xi\in\Xi$
\begin{align}
\hspace{-0.3cm}\max\{\|\nabla f(x)\|, \|G(x;\xi)\|  \} \leq M.\label{assump_bounded_gradient:Ineq}
\end{align}
\end{ass}
In this section, we assume that Assumptions \ref{assump:ergodicity}--\ref{assump:bounded_gradient} always hold. Given $\alpha_{k}$, let $\Gamma_{k}$ be defined as 
\begin{align}
\Gamma_{k} = \left\{\begin{array}{ll}
1,     &  k\leq 1\\
(1-\alpha_{k})\Gamma_{k-1}     & k\geq 2.
\end{array}\right.\label{notation:Gamma}
\end{align}
We first consider the following key lemma, which is essential in the analysis of Theorem \ref{thm:nonconvex} below. 
\begin{lem}\label{lem:nonconvex}
Let $\{\gamma_{k},\beta_{k}\}$ be nonnegative and nonincreasing and $\beta_{k}\leq \gamma_{k}$. Then 
\begin{align}
& \Eset[f(x_{k})] \leq \Eset[f(x_{k-1})] -\gamma_{k}\left(1 - L\gamma_{k}\right)\Eset\left[\|\nabla f(y_{k})\|^2\right] + 8LM^2\tau(\gamma_{k})\gamma_{k-\tau(\gamma_{k})}\gamma_{k} \notag\\
& \quad + (4M^2L + M)\gamma_{k}^2 + \frac{M^2 L\Gamma_{k}}{2}\sum_{t=1}^{k}\frac{(\gamma_{t}-\beta_{t})^2}{\Gamma_{t}\alpha_{t}} +  2LM^3 \gamma_{k}\sum_{t=k-\tau(\gamma_{k})}^{k-1}\alpha_{t}\Gamma_{t}\sum_{t=1}^{k}\frac{\gamma_{t}-\beta_{t}}{\Gamma_{t}}. \label{lem_nonconvex:Ineq}
\end{align}
\end{lem}
\begin{proof}[Sketch of proof.]
A complete analysis of this lemma is presented in the supplementary material. Here, we briefly discuss the main technical challenge in our analysis due to Markov samples, that is, the gradient samples are biased and dependent. First, using Assumption \ref{ass:Lipschitz} with some standard manipulation gives
\begin{align*}
f(x_{k}) &\leq f(x_{k-1}) -\gamma_{k}\left(1 - L\gamma_{k}\right) \|\nabla f(y_{k})\|^2 + 4M^2 L\gamma_{k}^2\\ 
& \quad + \frac{M^2 L\Gamma_{k}}{2}\sum_{t=1}^{k}\frac{(\gamma_{t}-\beta_{t})^2}{\Gamma_{t}\alpha_{t}} - \gamma_{k}\langle\nabla f(x_{k-1}),\,G(y_{k};\xi_{k}) - \nabla f(y_{k})\rangle.
\end{align*}
In the i.i.d settings, since the gradient samples are unbiased and independent, the last term on the right-hand side has a zero expectation. However, in our setting this expectation is different to zero and the samples are dependent. We, therefore, cannot apply the existing techniques to show \eqref{lem_nonconvex:Ineq}. Our key technique to address this challenge is to utilize the geometric mixing time $\tau$ defined in \eqref{notation:tau}. In particular, although the “noise” in our
algorithm is Markovian, its dependence is very weak at samples spaced out at every $\tau$ step. We, therefore, carefully characterize the progress of the algorithm in every $\tau$ step, resulting to the sum over $\tau$ steps on the right-hand side of \eqref{lem_nonconvex:Ineq}.
\end{proof}
To show our result for smooth nonconvex problems, we adopt the randomized stopping rule in \cite{GhadimiL2013}, which is common used in nonconvex optimization. In particular, given a sequence $\{y_{k}\}$ generated by Algorithm~\ref{alg:ASGD_nonconvex} we study the convergence on $y_{R}$, a point randomly selected from this sequence (a.k.a \eqref{thm_nonconvex:pk}). The convergence rate of Algorithm \ref{alg:ASGD_nonconvex} in solving problem \eqref{prob:main} is  stated as follows.   

\begin{thm}\label{thm:nonconvex}
Let $K>0$ be an integer such that 
\begin{align*}
&\alpha_{k} = \frac{2}{k+1},\quad \gamma_{k} \in \left[\beta_{k}\,,\, (1+\alpha_{k})\beta_{k}\right], \\ & \beta_{k} = \beta =  \frac{1}{\sqrt{K}}\leq \frac{1}{4L},\quad \forall k\geq 1.
\tagthis \label{thm_nonconvex:stepsizes}
\end{align*}
In addition, let $y_{R}$ be randomly selected from the sequence $\{y_{k}\}_{k=1}^{K}$ with probability $p_{k}$ defined as
\begin{align}
p_{k} = \frac{\gamma_{k}(1-L\gamma_{k})}{\sum_{k=1}^{K}\gamma_{k}(1-L\gamma_{k})}\cdot    
\label{thm_nonconvex:pk}
\end{align}
Then $y_{R}$ returned by Algorithm \ref{alg:ASGD_nonconvex} satisfies\footnote{Note that the same rate can be achieved for the quantity $\min_{k}\Eset\left[\|\nabla f(y_{k})\|^2\right]$.}
\begin{align*}
& \Eset\left[\|\nabla f(y_{R})\|^2\right]
\leq \frac{2(\Eset[f(x_{0})] - f^*)\left(4L+\sqrt{K}\right)}{K} + \frac{2M(LM^2(11+16C\log(K))+2)}{\sqrt{K}}. \tagthis    \label{thm_nonconvex:Ineq}
\end{align*}
\end{thm}
\begin{remark}
The rate in Theorem \ref{thm:nonconvex} matches the one presented in \cite{GhadimiL2013} for independent and unbiased gradient samples within a logarithmic factor, which captures the mixing time $\tau$.
\end{remark}
\begin{proof}[Proof of Theorem \ref{thm:nonconvex}]
Using \eqref{notation:Gamma} and \eqref{thm_nonconvex:stepsizes} yields $\Gamma_{k} = 2/k(k+1)$. Thus, using the integral test gives
\begin{align*}
\gamma_{k}\sum_{t=k-\tau(\gamma_{k})}^{k}\alpha_{t}\Gamma_{t} & = \gamma_{k}\sum_{t=k-\tau(\gamma_{k})}^{k}\frac{4}{t(t+1)^2} \leq \frac{2\gamma_{k}}{(k-\tau(\gamma_{k}))^2}. \tagthis \label{thm_nonconvex:Eq0}
\end{align*}
Next, using \eqref{thm_nonconvex:stepsizes} and $\Gamma_{k} = 2/k(k+1)$ we consider
\begin{align*}
\sum_{k=1}^{K}\Gamma_{k}\sum_{t=1}^{k}\frac{(\gamma_{t}-\beta_{t})^2}{\Gamma_{t}\alpha_{t}} & = \sum_{t=1}^{K}\frac{(\gamma_{t}-\beta_{t})^2}{\Gamma_{t}\alpha_{t}}\sum_{k=t}^{K}\Gamma_{k} = \sum_{t=1}^{K}\frac{(\gamma_{t}-\beta_{t})^2}{\Gamma_{t}\alpha_{t}}\sum_{k=t}^{K}\frac{2}{k(k+1)}\notag\\ 
&= \sum_{t=1}^{K}\frac{2(\gamma_{t}-\beta_{t})^2}{\Gamma_{t}\alpha_{t}}\sum_{k=t}^{K}\left(\frac{1}{k}-\frac{1}{k+1}\right)\\
&\leq \sum_{t=1}^{K}\frac{2(\gamma_{t}-\beta_{t})^2}{\Gamma_{t}\alpha_{t}}\frac{1}{t}
\leq 2\sum_{t=1}^{K}\frac{\beta_{t}^2\alpha_{t}^2}{t\Gamma_{t}\alpha_{t}} = 2\beta^2K. \tagthis \label{thm_nonconvex:Eq1a}
\end{align*}
Similarly, using \eqref{thm_nonconvex:Eq0}, $\alpha_{k}\leq 1$, $\gamma_{k}\leq 2\beta_{k} = 2\beta$ we have
\begin{align*}
\sum_{k=1}^{K}\gamma_{k}\sum_{t=k-\tau(\gamma_{k})}^{k}\alpha_{t}\Gamma_{t}\sum_{t=1}^{k}\frac{\gamma_{t}-\beta_{t}}{\Gamma_{t}} & \leq  \sum_{k=1}^{K}\frac{2\gamma_{k}}{(k-\tau(\gamma_{k}))^2}\sum_{t=1}^{k}\frac{\gamma_{t}-\beta_{t}}{\Gamma_{t}}\notag\\
&\leq \sum_{t=1}^{K}\frac{2\beta^2\alpha_{t}}{\Gamma_{t}}\sum_{k=t}^{K}\frac{1}{(k-\tau(\gamma_{k}))^2} \\ 
& \leq2\beta^2\sum_{t=1}^{K}\frac{t}{t-\tau(\gamma_{t})}\leq 4\beta^2K. \tagthis \label{thm_nonconvex:Eq1b}
\end{align*}
Moreover, using \eqref{notation:tau} and $\gamma_{k}\geq\beta_{k} = \beta$ we have
$\tau(\gamma_{k})=  C\log\left(\frac{1}{\gamma_{k}}\right) \leq C\log\left(1/\beta\right)$, which gives
\begin{align*}
\sum_{k=1}^{K}\tau(\gamma_{k})\gamma_{k-\tau(\gamma_{k})}\gamma_{k} & \leq 4C\sum_{k=1}^{K}\beta_{k}\beta_{k-\tau(\gamma_{k})}\log(k) = 2C\beta^2K\log(K). \tagthis \label{thm_nonconvex:Eq1c}  
\end{align*}
Since $\alpha_{k}\leq 1$ for $k\geq 1$ we have
$\sum_{k=1}^{K}\gamma_{k}^2 \leq 2\beta^2K.$ We now use the relations \eqref{thm_nonconvex:Eq1a}--\eqref{thm_nonconvex:Eq1c} to derive \eqref{thm_nonconvex:Ineq}. Indeed, summing up both sides of \eqref{lem_nonconvex:Ineq} over $k$ from $1$ to $N$ and reorganizing yield
\begin{align*}
&\quad \sum_{k=1}^{K}\gamma_{k}\left(1 - L\gamma_{k}\right)\Eset\left[\|\nabla f(y_{k})\|^2\right] \notag\\
&\leq \Eset[f(x_{0})] - \Eset[f(x_{K})] + (4LM^2+M)\sum_{k=1}^{K}\gamma_{k}^2
+ 8LM^2\sum_{k=1}^{K}\tau(\gamma_{k})\gamma_{k-\tau(\gamma_{k})}\gamma_{k}\notag\\
&\quad + \sum_{k=1}^{K}\frac{M^2 L\Gamma_{k}}{2}\sum_{t=1}^{k}\frac{(\gamma_{t}-\beta_{t})^2}{\Gamma_{t}\alpha_{t}}  + 2M^3L\sum_{k=1}^{K}\gamma_{k}\sum_{t=k-\tau(\gamma_{k})}^{k}\alpha_{t}\Gamma_{t}\sum_{t=1}^{k}\frac{(\gamma_{t}-\beta_{t})}{\Gamma_{t}}\notag\\
&\leq (\Eset[f(x_{0})] - f^*) + (11LM^2 + 2M)\beta^2K +  16CLM^2\beta^2K\log(K), \tagthis \label{thm_nonconvex:Eq1}
\end{align*}
where we use $\Eset[f(x_{K})] \geq f^*$ and $\beta\leq 1/4L$. Dividing both sides by $\sum_{k=1}^{K}\gamma_{k}(1-L\gamma_{k})$ gives
\begin{align*}
\frac{\sum_{k=1}^{K}\gamma_{k}\left(1 - L\gamma_{k}\right)\Eset\left[\|\nabla f(y_{k})\|^2\right]}{\sum_{k=1}^{K}\gamma_{k}(1-L\gamma_{k})}
& \leq \frac{(\Eset[f(x_{0})] - f^*) }{\sum_{k=1}^{K}\gamma_{k}(1-L\gamma_{k})} + \frac{(11LM^2+2M)\beta^2K }{\sum_{k=1}^{K}\gamma_{k}(1-L\gamma_{k})} \notag\\ 
&\quad  + \frac{16CLM^2\beta^2K\log(K)}{\sum_{k=1}^{K}\gamma_{k}(1-L\gamma_{k})}\cdot
\end{align*}
Using \eqref{thm_nonconvex:stepsizes} yields $1-L\gamma_{k}\geq 1/2$ and $\sum_{k=1}^{K}\gamma_{k}(1-L\gamma_{k})\geq \sum_{k=1}^{K} \beta_{k}/2 = K\beta/2$. Thus, we obtain
\begin{align*}
\frac{\sum_{k=1}^{K}\gamma_{k}\left(1 - L\gamma_{k}\right)\Eset\left[\|\nabla f(y_{k})\|^2\right]}{\sum_{k=1}^{K}\gamma_{k}(1-L\gamma_{k})}
&\leq \frac{2(\Eset[f(x_{0})] - f^*) }{K\beta} + \frac{2(11LM^2+2M)\beta^2K }{K\beta}  \\ & \quad + \frac{32CLM^2\beta^2K\log(K)}{K\beta}\notag\\
&\leq \frac{2(\Eset[f(x_{0})] - f^*)\left(4L+\sqrt{K}\right)}{K} \\ & \quad + \frac{2(11LM^2+2M)}{\sqrt{K}} + \frac{32CLM^2\log(K)}{\sqrt{K}},
\end{align*}
which by using \eqref{thm_nonconvex:pk} gives \eqref{thm_nonconvex:Ineq}.
\end{proof}

\section{Convergence analysis: Convex case}
\label{sec:convex}
\begin{algorithm}[h]
\caption{{\sf AMGD} for convex problems}
  {\bfseries Initialize:} Set arbitrarily $x_{k},\xbar_{k}\in\Xcal$, step sizes $\{\alpha_{k},\beta_{k},\gamma_{k}\}$ for $k\leq 1$, and an integer $K\geq 1$\vspace{0.2cm}\\
  {\bfseries Iterations:} For $k = 1,\ldots,K$ do
   \begin{align}
    y_{k}  &= (1-\beta_{k})\xbar_{k-1} + \beta_{k}x_{k-1}\label{alg:y}\\
    v_{k} & = \gamma_{k}\left[\langle G(y_{k};\xi_{k})\,,\,x - y_{k}\rangle + \mu V(y_{k},x)\right] \\
    x_{k} &= \arg\min_{x\in\Xcal}\Big\{v_k + V(x_{k-1},x)\Big\}\label{alg:x} \\
    \xbar_{k} &= (1-\alpha_{k})\xbar_{k-1} + \alpha_{k}x_{k}\label{alg:xbar}
    \end{align}
   {\bfseries Output:} $\xbar_{k}$    
\label{alg:ASGD_convex}
\end{algorithm}
In this section, we study Algorithm \ref{alg:ASGD_convex} for solving  \eqref{prob:main} when $f$ is convex and $\Xcal$ is compact.  For simplicity we consider $V$ in Algorithm \ref{alg:ASGD_convex} is the Euclidean distance, i.e., $\psi(x) = \frac{1}{2} \|x\|^2$ and $V(y,x) = \frac{1}{2}\|y-x\|^2$, although our results can be extended to the general Bregman distance $V$. Since $\Xcal$ is compact, there exist $D,M>0$
\begin{align}
\hspace{-0.3cm}D = \max_{x\in\Xcal}\|x\|,\;\; \|G(x;\xi)\|\leq M,\;\forall \xi\in\Xi,\; \forall x\in\Xcal.     \label{notation:DM}
\end{align}
In addition, let 
$x^* = \arg\min_{x\in\Xcal} f(x).$ We assume that  $\{\alpha_{k},\beta_{k}\}$ are chosen such that $\alpha_{1} = 1$ and 
\begin{align}
\frac{\beta_{k}(1-\alpha_{k})}{\alpha_{k}(1-\beta_{k})} = \frac{1}{1+\mu\gamma_{k}},\qquad 1+\mu\gamma_{k}  > L\alpha_{k}\gamma_{k}, \label{condition:stepsizes}
\end{align}
where $\mu \geq 0$ and $L$ is given in \eqref{ass_Lipschitz:Ineq1}. The key idea to derive the results in this section is to utilize Assumption \ref{assump:ergodicity} to the handle the Markovian ``noise", similar to the one in Section \ref{sec:nonconvex}. For an ease of exposition, we present the analysis of the results in this section to the supplementary material.

We now study the rates of Algorithm \ref{alg:ASGD_convex} when the function $f$ is only convex and  Assumption \ref{ass:Lipschitz} holds. In this case $\partial f(\cdot) = \nabla f(\cdot)$ and $\mu = 0$. Since $\mu = 0$, Eq.\ \eqref{condition:stepsizes} gives $\beta_{k} = \alpha_{k}$ and $y_{k} = \xbar_{k}$. The convergence rate of Algorithm \ref{alg:ASGD_convex} in this case is given below. 
\begin{thm}[Convex functions]\label{thm:convex}
Let Assumptions \ref{assump:ergodicity}--\ref{ass:Lipschitz} hold. Suppose that the step sizes are chosen as
\begin{align}
\alpha_{k} = \frac{2}{k+1},\quad \gamma_{k} = \frac{1}{2L\sqrt{k+1}}\cdot\label{thm_convex:stepsizes}
\end{align}
Then we have for all $k\geq 1$
\begin{align*}
f(\xbar_{k}) - f(x^*) & \leq \frac{f(\xbar_{0})+4LD}{2k(k+1)} + \frac{2(D+2M^2)}{3L\sqrt{k}} + \frac{2(4D^2L + M^2)(L+1)\log(2L\sqrt{k})}{\sqrt{k}}. \tagthis \label{thm_convex:Ineq}
\end{align*}
\end{thm}
Finally, we now provide the rates of Algorithm \ref{alg:ASGD_convex} when $f$ is strongly convex.  
\begin{ass}\label{ass:sc}
There exists a constant $\mu > 0$ s.t. $\forall x,y$ and $g(x)\in\partial f(x)$ we have
\begin{align}
\frac{\mu}{2}\|y-x\|^2\leq f(y) - f(x) - \langle g(x), y - x\rangle.\label{ass_sc:Ineq}
\end{align}
\end{ass}
\begin{thm}[Strong convexity]\label{thm:sc}
Suppose that Assumptions \ref{assump:ergodicity}, \ref{ass:Lipschitz}, and \ref{ass:sc} hold. Let the step sizes be
\begin{align*}
\alpha_{k} = \frac{2}{k+1},\quad \gamma_{k} = \frac{2}{\mu k},\quad \beta_{k} = \frac{\alpha_{k}}{\alpha_{k}+(1-\alpha_{k})(1+\mu\gamma_{k})}. \tagthis \label{thm_sc:stepsizes}
\end{align*}
Then we have for all $k\geq 1$
\begin{align*}
f(\xbar_{k}) - f(x^*) &\leq \frac{2f(\xbar_{0}) + 6\mu D}{k(k+1)} + \frac{4(M^2 +2\mu MD+12\mu LD^2)
(2+\mu)\log(\frac{\mu(k+1)}{2})}{\mu k}\\
& \quad + \frac{2D + 10M^2 + 8\mu MD}{\mu(k+1)}\cdot \tagthis \label{thm_sc:Ineq}
\end{align*}
\end{thm}

\begin{remark}
We note that in Theorem \ref{thm:convex}, {\sf ASGD} has the same worst case convergence rate as compared to {\sf SGD}, i.e., $\Ocal(1/\sqrt{k})$. However, {\sf ASGD} has much better rate on the initial condition than {\sf SGD}, i.e., $\Ocal(1/k^2)$ versus $\Ocal(1/k)$. Similar observation holds for Theorem \ref{thm:sc}. This gain is very important, e.g., in improving the data efficiency of RL algorithms as illustrated in Section \ref{sec:simulation:TD}, where we study temporal difference methods.    
\end{remark}

\section{Numerical experiments}\label{sec_experiment}
In this section, we apply the proposed accelerated Markov gradient methods for solving a number of problems in {\sf RL}, where the samples are taken from Markov processes. In particular, we consider the usual setup of {\sf RL} where the environment is modeled by a Markov decision process ({\sf MDP}) \cite{Sutton2018}. Let $\Scal$ and $\Acal$ be the (finite) set of states and action. We denote by $\pi_{\theta}(s,a) = Pr(a_{k} = a|s_{k} = s,\theta)$ the randomized policy parameterized by $\theta$, where $s\in\Scal$ and $a\in\Acal$. The goal is to find $\theta$ to maximize 
\begin{align*}
f(\pi_{\theta})  \triangleq \Eset_{s_{0}}\left[V^{\pi_{\theta}}(s_{0})\right]\quad \text{where}\quad V^{\pi_{\theta}}(s_{0})  \triangleq  \Eset\left[\sum_{k=0}^{\infty}\gamma^{k}r_{k}\,|\,s_{0},\pi_{\theta}\right],
\end{align*}
$\gamma$ is the discounted factor and $r_{k}$ is the reward returned by the environment at time $k$. We study the accelerated variants of on-policy temporal difference learning (using Algorithm \ref{alg:ASGD_convex}) and Monte-Carlo policy gradient (REINFORCE) methods  (using Algorithm \ref{alg:ASGD_nonconvex}), and compare their performance with the classic (non-accelerated) counterparts. In all the experiments we consider, the proposed accelerated variants of RL algorithms outperform the classic ones, which agrees with our theoretical results in Theorems \ref{thm:nonconvex} and \ref{thm:convex}. 



\begin{figure}[t]
\begin{center}
\includegraphics[width = 0.49\textwidth]{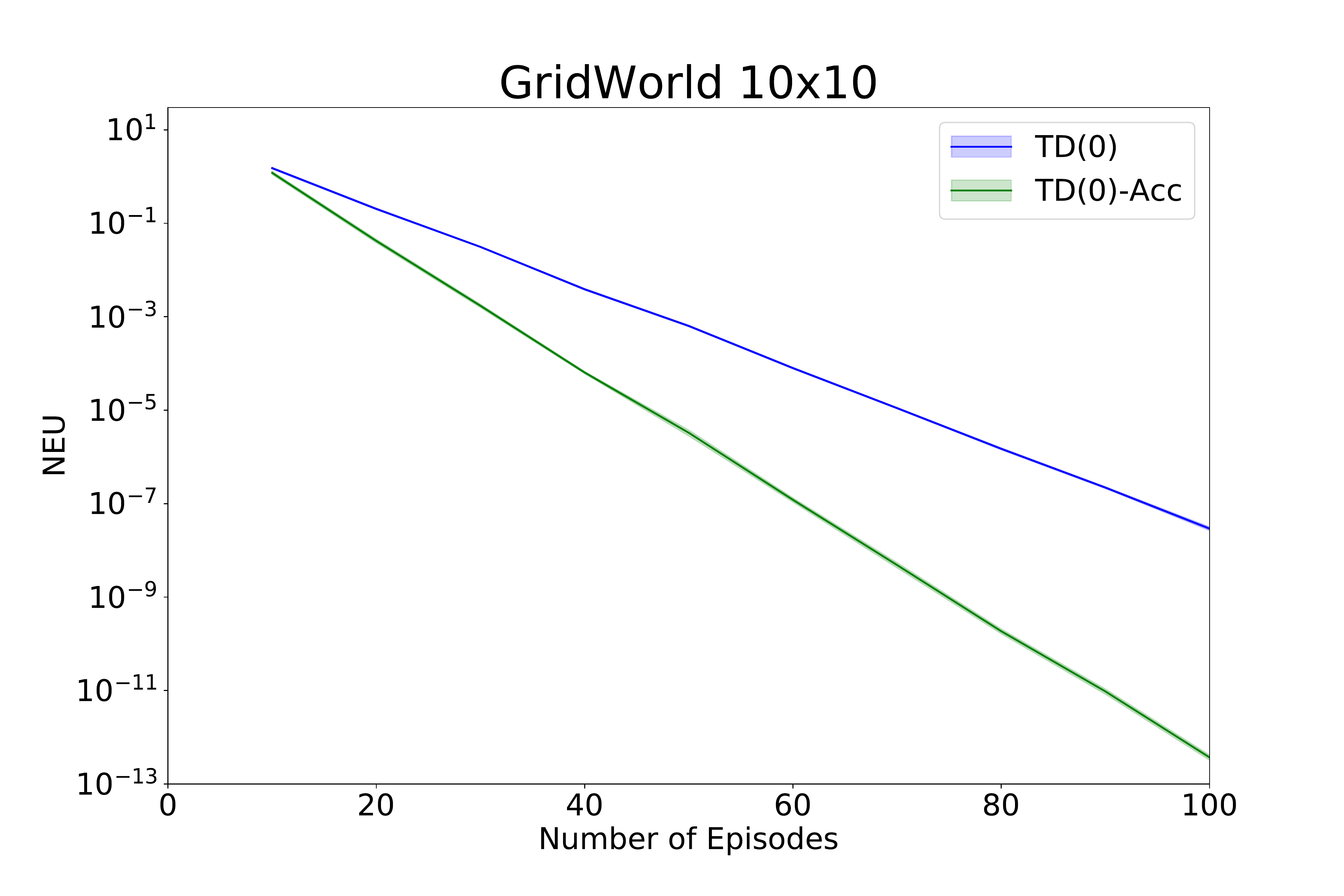}
\includegraphics[width = 0.49\textwidth]{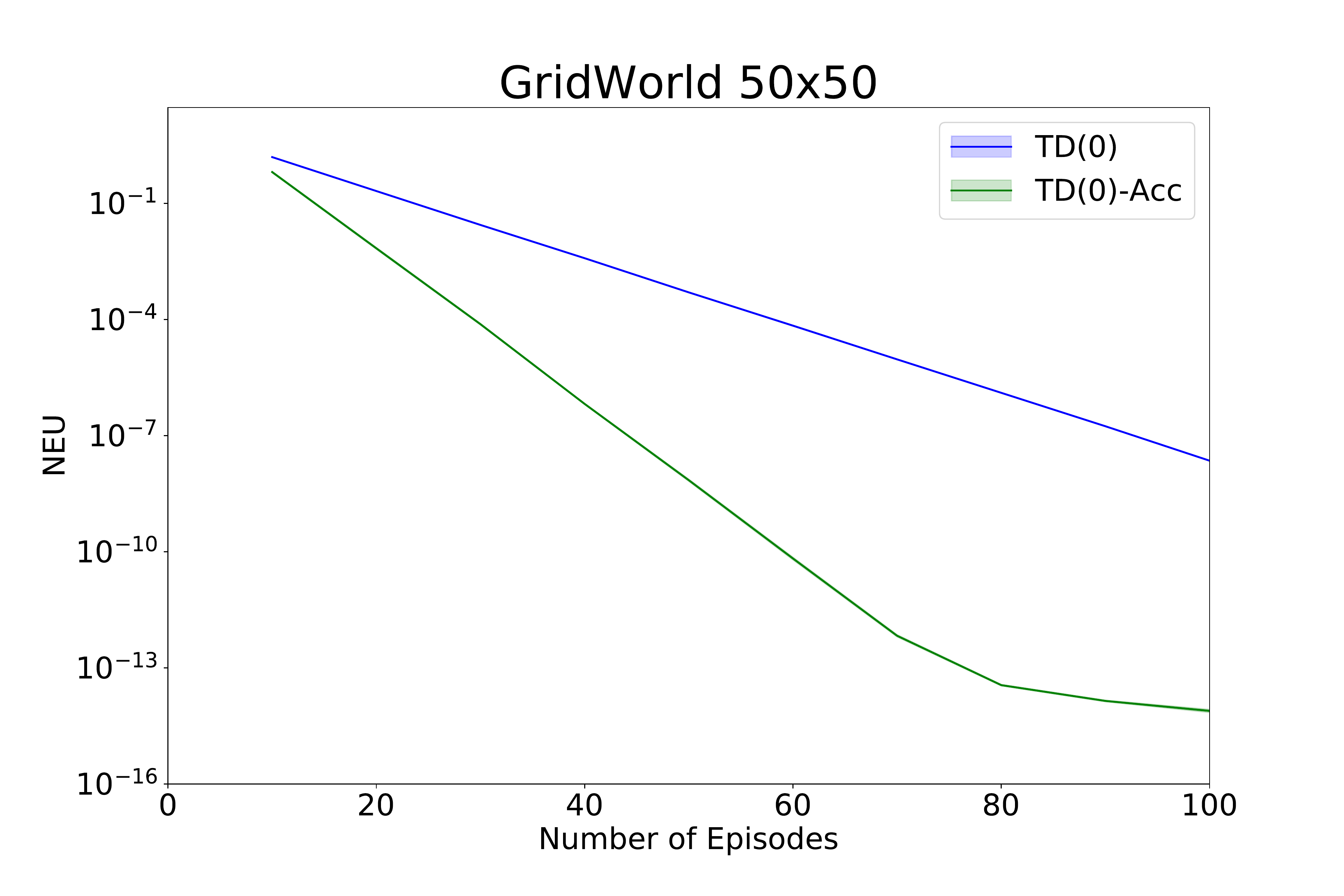}
\caption{The performance of two TD(0) variants on GridWorld with sizes $10\times10$ (left) and $50\times50$ (right).}
\label{fig:gridworld}
\end{center}
\end{figure}

\textbf{General setup:} For each simulation, we run the algorithm 10 times with the same initial policy and record the performance measures. The performance of each algorithm is specified by averaging the metric over the number of episodes. Here, an episode is defined as the set of state-action pairs collected from beginning until the terminal state or a specified episode is reached. The plots consist of the mean with 90\% confidence interval (shaded area) of the performance metric. For REINFORCE-based methods, we randomly generate an initial policy represented by a neural network.


\subsection{Accelerated temporal difference learning}\label{sec:simulation:TD}
One of the central problems in {\sf RL} is the so-called policy evaluation problem, where we want to estimate the vector value function $V^{\pi_{\theta}}$ for a fixed policy $\pi_{\theta}$. Temporal difference learning ({\sf TD}($\lambda$)), originally proposed by Sutton \cite{Sutton1988_TD}, is one of the most efficient and practical methods for policy evaluation. It is shown in \cite{Ollivier2018} that if the underlying Markov process is reversible, {\sf TD} learning is a gradient descent method. In addition, under linear function approximation this {\sf TD} method can be viewed as gradient descent for solving a strongly convex quadratic problem \cite{Tsitsiklis1997_TD}. In this problem, the data tuple $\{s_{k},a_{k},s_{k+1}\}$ generated by the {\sf MDP} is $\xi_{k}$ in our model; see \cite{Tsitsiklis1997_TD} for more details.  

For our simulation, we consider the policy evaluation problem over the GridWorld environment \cite[Example 4.1]{Sutton2018}, where the agent is placed in a grid and wants to reach a goal from an initial position. The starting and goal positions are fixed at the top-left and bottom-right corners, respectively. We implement the one-step {\sf TD} (or {\sf TD}(0)), and apply our framework to obtain its accelerated variant, denoted as TD(0)-Acc. The value function is approximated by using linear function approximation, i.e., $V^{\pi_{\theta}}(s) = \langle \theta,\Phi(s) \rangle$ where $\Phi(s)$ is the feature at $s\in\Scal$ using three   Fourier basis vectors \cite{konidaris2011}. We consider a randomized policy choosing action uniformly over the set $\{up,down,left,right\}$. In this case, the transition matrix is doubly stochastic, therefore, reversible with uniform distribution.

We use $\gamma =  0.9$ in all GridWorld environments and consider the episodic version of TD(0), where the episode length is 10 times the grid size, i.e. episode length is 100 for a $10\times10$ grid. The learning rate of TD(0) is set to 0.001 while we set the stepsizes $\beta_k$ and $\alpha_k$ for TD(0)-Acc as in  \eqref{thm_sc:stepsizes}, and $\gamma_k = \frac{2\delta}{\mu(k+1)}$ with $\delta = 0.1$. Due to the episodic nature, the stepsizes of both methods are adapted at the episodic level.

Since the optimal solution is unknown, we use the norm of expected {\sf TD} update (NEU) as in \cite{sutton2009} to compare the performance of {\sf TD}$(0)$ and {\sf TD}$(0)$-Acc.  In each run, after every 10 episodes, the NEU is computed by averaging over 10 test episodes. The performance of both TD(0) variants the gridworld environment with size $10\times10$ and $50\times50$ are presented in Figure~\ref{fig:gridworld}, which shows that the proposed method, {\sf TD}$(0)$-Acc, outperforms the classic {\sf TD}$(0)$. 



\subsection{Accelerated REINFORCE methods}

The REINFORCE method can be viewed as {\sf SGD} in reinforcement learning \cite{Williams1992}.  To evaluate the two variants of REINFORCE, we consider five different control problems, namely, Acrobot, CartPole,  Ant, Swimmer, and HalfCheetah, using the simulated environments from OpenAI Gym and Mujoco \cite{OpenAIGym, todorov2012mujoco} . We utilize the implementation of REINFORCE from \texttt{rllab} library \cite{Duan_rllab}. 
More details of these environments along with the size of observation and action spaces are given in Table \ref{tab:environments}.

\begin{figure}[h]
\begin{center}
    \includegraphics[width = 0.49\textwidth]{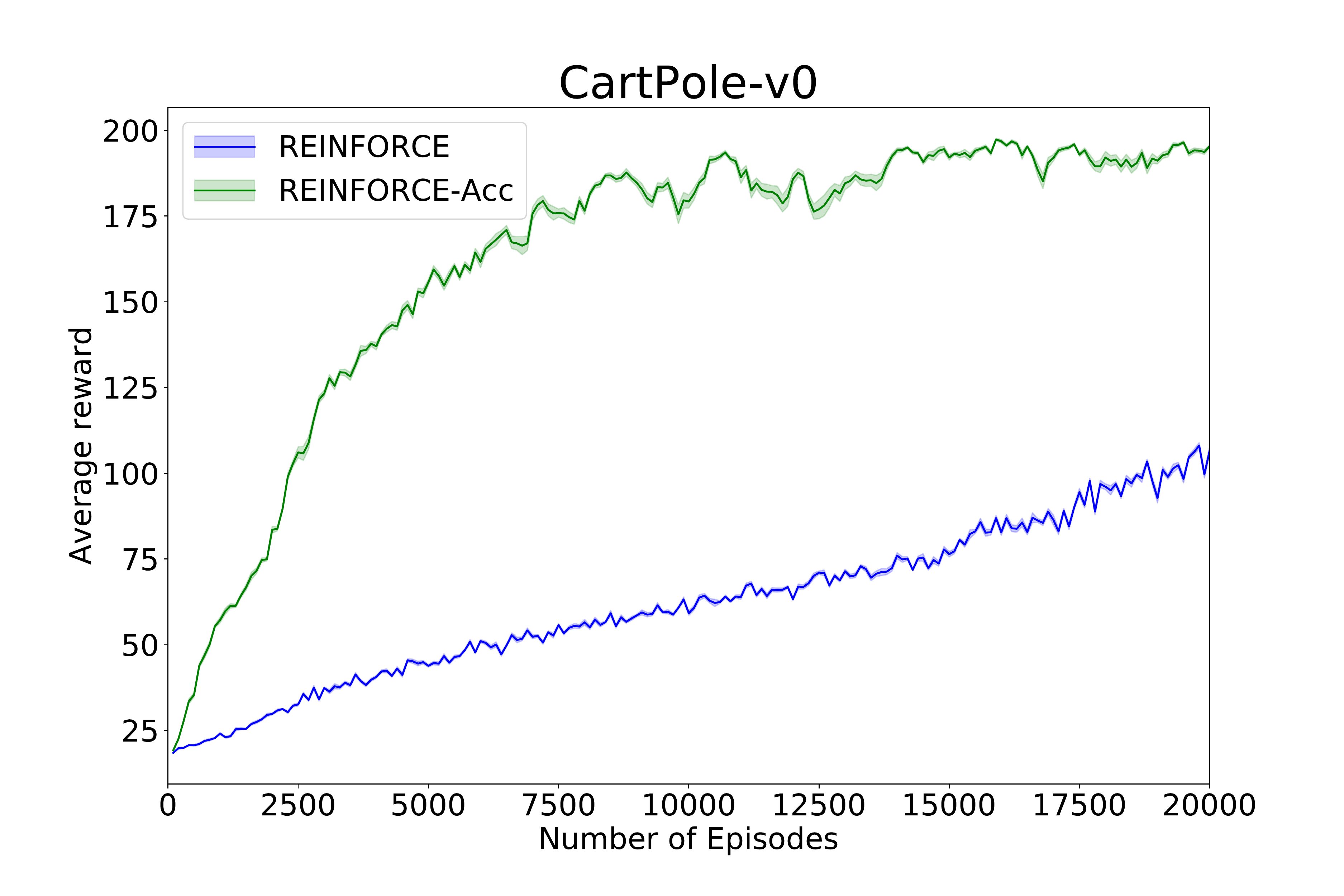}
    \includegraphics[width = 0.49\textwidth]{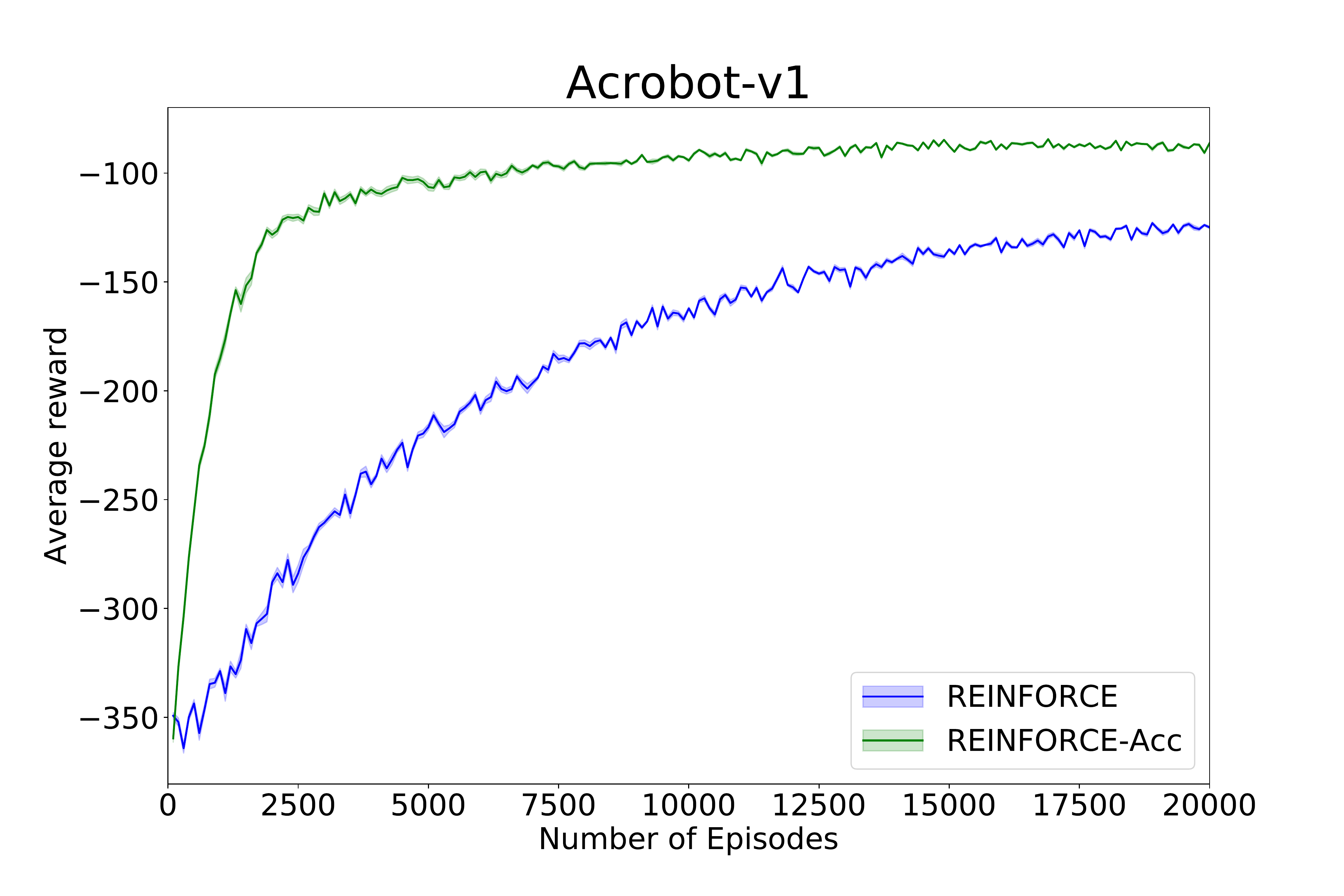}
\caption{The performance of two algorithms on the \texttt{CartPole-v0} and \texttt{Acrobot-v1} environment.}
\label{fig:acrobot_cartpole}
\label{fig:acrobot}
\end{center}
\end{figure}

\textbf{Brief summary}: At every iteration, we collect a batch of episodes with different length depending on the environment. We then update the policy parameters and record the performance measure by collecting 50 episodes using the updated policy and average the total rewards for all episodes.

We first compare the algorithms using discrete control tasks: \texttt{Acrobot-v1} and \texttt{CartPole-v0} environments. 
For these discrete tasks, we use a soft-max policy $\pi_\theta$ with parameter $\theta$ defined as
\begin{equation*}
    \pi_\theta(a|s) = \frac{e^{\phi_(s,a,\theta)}}{\sum_{k=1}^{\vert\mathcal{A}\vert} e^{\phi(s,a_k,\theta)}},
\end{equation*}
where $\phi(s,a,\theta)$ is represented by a neural network and $\vert\mathcal{A}\vert$ is the total number of actions. Figure \ref{fig:acrobot_cartpole} presents the performance of two algorithms on these environments, where the accelerated REINFORCE significantly outperforms its non-accelerated variant.

\begin{table*}[htp!]
\caption{Descriptions of environments used for numerical experiments.}
\label{tab:environments}
\resizebox{\textwidth}{!}{%
\begin{tabular}{|l|c|c|c|l|}
\hline
\textbf{Environment} & \textbf{Obs. Space} & \textbf{Action Space} & \textbf{Action Type} & \multicolumn{1}{c|}{\textbf{Descriptions}} \\ \hline
\multirow{3}{*}{Acrobot-v1} & \multirow{3}{*}{2} & \multirow{3}{*}{3} & \multirow{3}{*}{Discrete} & The Acrobot-v1 environment contains two joints and two links where we can actuate \\
 &  &  &  & the joints between two links. The links are hanging downwards at the beginning \\
 &  &  &  & and the goal is to swing the end of the lower link up to a given height. \\ \hline
\multirow{2}{*}{CartPole-v0} & \multirow{2}{*}{4} & \multirow{2}{*}{2} & \multirow{2}{*}{Discrete} & A pole is attached by an un-actuated joint to a cart moving along a frictionless track. \\
 &  &  &  & The pole starts upright, and the goal is to prevent it from falling over. \\ \hline
Swimmer-v2 & 8 & 2 & Continuous & The goal is to make a four-legged creature walk forward as fast as possible. \\ \hline
Walker2d-v2 & 8 & 2 & Continuous & The goal is to make a two-dimensional bipedal robot walk forward as fast as possible. \\ \hline
HalfCheetah-v2 & 17 & 6 & Continuous & Make a two-legged creature move forward as fast as possible. \\ \hline
Ant-v2 & 111 & 8 & Continuous & Make a four-legged creature walk forward as fast as possible. \\ \hline
\end{tabular}%
}
\end{table*}

\begin{table*}[htp!]
\caption{Parameters for \texttt{Acrobot-v1}, \texttt{CartPole-v0}, \texttt{Swimmer-v2}, \texttt{Walker2d-v2}, \texttt{HalfCheetah-v2}, and \texttt{Ant-v2} environments.}
\label{tab:params}
\resizebox{\textwidth}{!}{%
\begin{tabular}{|l|l|c|c|c|c|c|c|}
\hline
\textbf{Environment} & \textbf{Algorithm} & \multicolumn{1}{l|}{\textbf{Policy Network}} & \multicolumn{1}{l|}{\textbf{Discount Factor}} & \multicolumn{1}{l|}{\textbf{Episode Length}} & \multicolumn{1}{l|}{\textbf{Baseline}} & \multicolumn{1}{l|}{\textbf{Batch Size}} & \multicolumn{1}{l|}{\textbf{Learning Rate}} \\ \hline
\multirow{2}{*}{CartPole-v0} & REINFORCE & \multirow{2}{*}{4x8x2} & \multirow{2}{*}{0.99} & \multirow{2}{*}{200} & \multirow{2}{*}{None} & 25 & 0.1 \\ \cline{2-2} \cline{7-8} 
 & REINFORCE-Acc &  &  &  &  & 25 & 0.1 \\ \hline
\multirow{2}{*}{Acrobot-v1} & REINFORCE & \multirow{2}{*}{6x16x3} & \multirow{2}{*}{0.99} & \multirow{2}{*}{500} & \multirow{2}{*}{None} & 25 & 0.1 \\ \cline{2-2} \cline{7-8} 
 & REINFORCE-Acc &  &  &  &  & 25 & 0.1 \\ \hline
\multirow{2}{*}{Swimmer-v2} & REINFORCE & \multirow{2}{*}{8x32x32x2} & \multirow{2}{*}{0.99} & \multirow{2}{*}{1000} & \multirow{2}{*}{Linear} & 100 & 0.01 \\ \cline{2-2} \cline{7-8} 
 & REINFORCE-Acc &  &  &  &  & 100 & 0.01 \\ \hline
 \multirow{2}{*}{Walker2d-v2} & REINFORCE & \multirow{2}{*}{17x32x32x6} & \multirow{2}{*}{0.99} & \multirow{2}{*}{1000} & \multirow{2}{*}{Linear} & 100 & 0.05 \\ \cline{2-2} \cline{7-8} 
 & REINFORCE-Acc &  &  &  &  & 100 & 0.05 \\ \hline
\multirow{2}{*}{HalfCheetah-v2} & REINFORCE & \multirow{2}{*}{17x32x32x6} & \multirow{2}{*}{0.99} & \multirow{2}{*}{1000} & \multirow{2}{*}{Linear} & 100 & 0.05 \\ \cline{2-2} \cline{7-8} 
 & REINFORCE-Acc &  &  &  &  & 100 & 0.05 \\ \hline
\multirow{2}{*}{Ant-v2} & REINFORCE & \multirow{2}{*}{111x128x64x32x8} & \multirow{2}{*}{0.99} & \multirow{2}{*}{1000} & \multirow{2}{*}{Linear} & 100 & 0.01 \\ \cline{2-2} \cline{7-8} 
 & REINFORCE-Acc &  &  &  &  & 100 & 0.01 \\ \hline
\end{tabular}%
}
\end{table*}

Next, we evaluate the performance of these algorithms on continuous control tasks in Mujoco. In these environments, we also incorporate a linear baseline in order to reduce the variance of the policy gradient estimator, see e.g. \cite{greensmith2004variance}. The actions are sampled from a deep Gaussian policy given as
$$\pi_\theta(a\vert s) = \mathcal{N}\left(\phi(s,a,\theta_\mu) ; \phi(s,a,\theta_\sigma) \right),$$
where $\phi(.)$ is the output from a neural network. The mean and variance of the Gaussian distribution is learned in this experiment. 
We present all parameters setup for all environments with REINFORCE variants in Table~\ref{tab:params}.
The network parameters are specified as $(\text{input} \times \text{hidden layers} \times \text{output})$, i.e. a $4\times 8\times 2$ network contains 1 hidden layer of 8 neurons.



We evaluate these algorithms on four environments with increasing difficulty: \texttt{Swimmer}, \texttt{Walker2d}, \texttt{HalfCheetah}, and \texttt{Ant}. Figure~\ref{fig:mujoco_envs} illustrates the results in those environments, where the accelerated variant REINFORCE-Acc indeed shows its advantage over REINFORCE in both discrete and continuous tasks which well aligns with our theoretical results.

\begin{remark}
It is also interesting to show the benefit of using Markov samples in  {\sf ASGD}. Since this benefit has been studied in \cite{SunSY2018} for {\sf SGD}, we skip this study in this paper. Our goal in this section is to draw some potential applications of acceleration techniques in stochastic optimization to {\sf RL} algorithms, which is the main motivation of this paper.   
\end{remark}

\begin{figure}[htp!]
\begin{center}
\includegraphics[width = 0.49\textwidth]{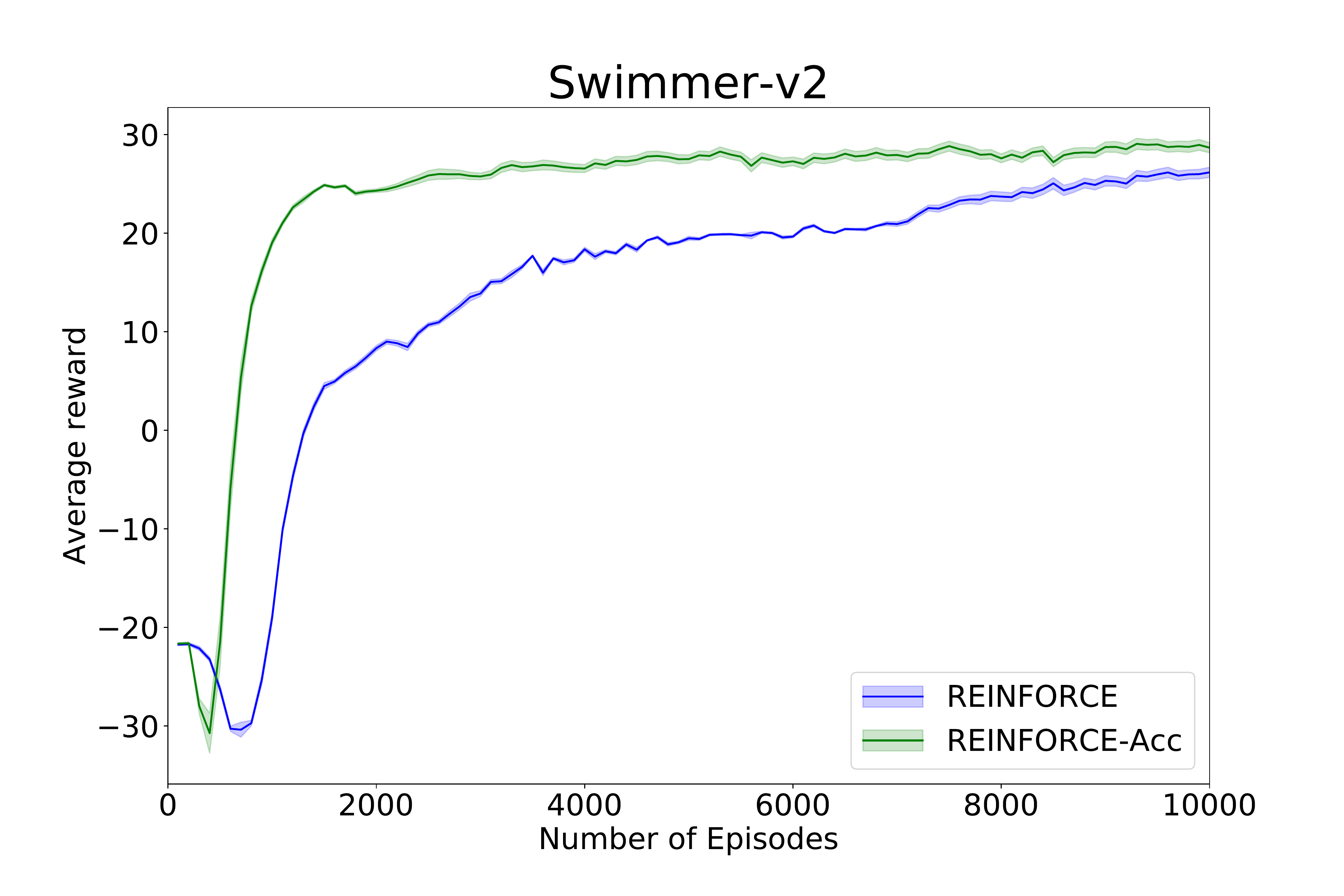}
\includegraphics[width = 0.49\textwidth]{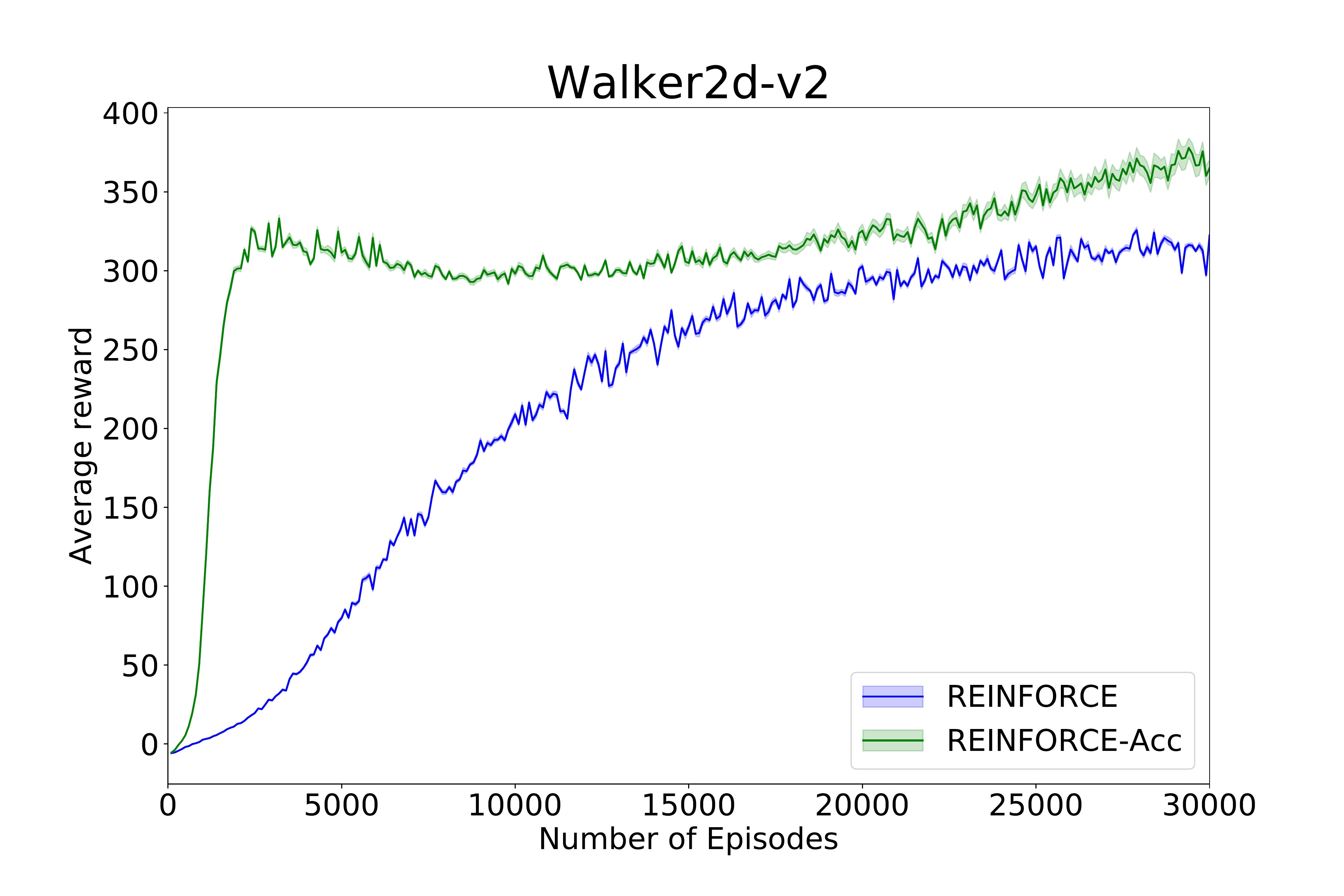}
\includegraphics[width = 0.49\textwidth]{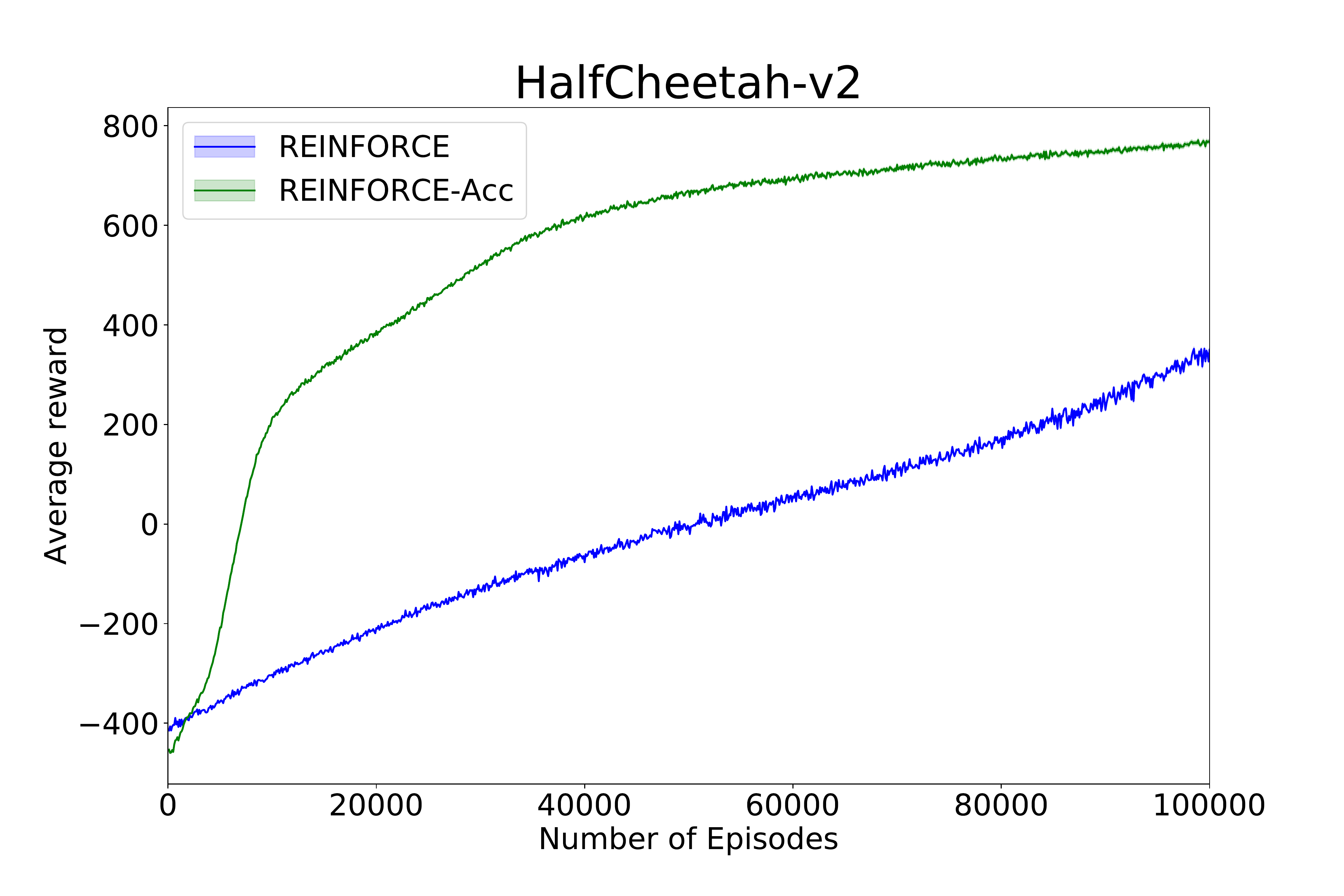}
\includegraphics[width = 0.49\textwidth]{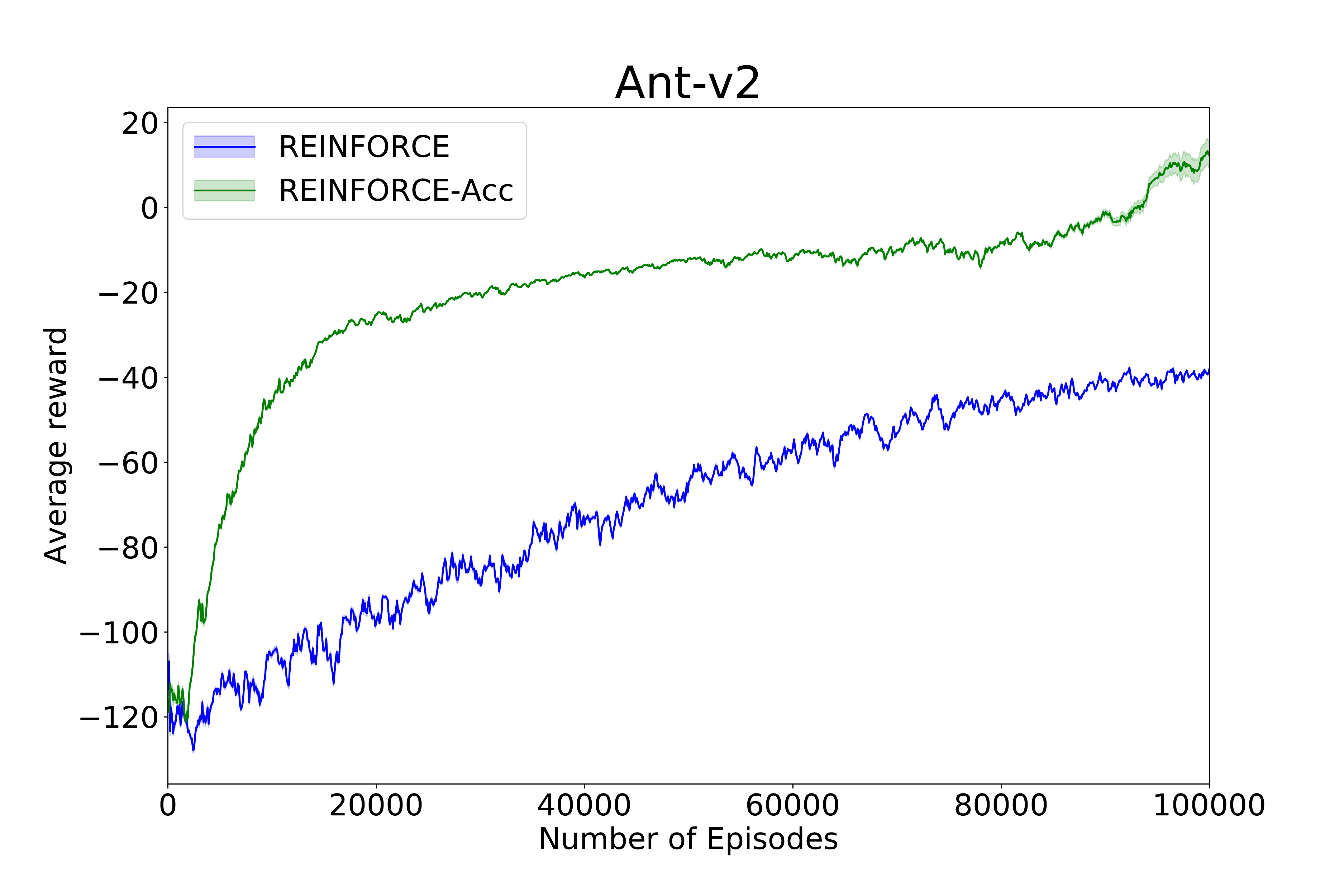}
\caption{The performance of two algorithms on 4 Mujoco environments.}
\label{fig:mujoco_envs}
\end{center}
\end{figure}

\section{Conclusion}
\label{sec_conclusion}

In this paper, we study a variant of {\sf ASGD} for solving (possibly nonconvex) optimization problems, when the gradients are sampled from Makrov process. We characterize the finite-time performance of this method when the gradients are sampled from Markov processes, which shows that {\sf ASGD} converges at the nearly the same rate with Markovian gradient samples as with independent gradient samples.  The only difference is a logarithmic factor that accounts for the mixing time of the Markov chain. We apply the accelerated methods to policy evaluation problems in GridWorld environment and to several challenging problems in the OpenAI Gym and Mujoco. Our simulations show that acceleration can significantly improve the performance of the classic RL algorithms. One future interesting directions is to relax the technical Assumption \ref{assump:ergodicity}. It is also possible to consider applying acceleration methods to other reinforcement learning algorithms.

\bibliography{reference}
\bibliographystyle{plain}

\appendix
\newpage

\section{Appendix}
We first state the following result, as a consequence of the geometric mixing time in Assumptions \ref{assump:ergodicity} and Lipschitz continuity in \ref{ass:Lipschitz}  and \ref{assump:bounded_gradient}. The proof of this result can be found in 
\cite[Lemma 3.2]{ChenZDMC2019}, therefore, it is omitted here for brevity. 
\begin{cor}
Suppose that Assumptions \ref{assump:ergodicity}, \ref{ass:Lipschitz}, \ref{assump:bounded_gradient} hold. Let $g(x) \in\partial f(x)$ and $\tau(\gamma)$ defined in \eqref{notation:tau}. Then
\begin{align}
&\|\Eset[G(x;\xi_{k})] - g(x)\,|\, \xi_{0} = \xi\| \leq \gamma,\quad \forall x,\; \forall k\geq \tau(\gamma).\label{appendix:cor:mixing:ineq}
\end{align} 
\end{cor}

\subsection{Proof of Lemma \ref{lem:nonconvex}}\label{appendix:lem_nonconvex}
\begin{proof}
Since $f$ satisfies Assumption \ref{ass:Lipschitz}, by  \eqref{alg_nonconvex:xbar} and \eqref{alg_nonconvex:x} we have
\begin{align*}
f(x_{k}) &\leq f(x_{k-1}) + \langle\nabla f(x_{k-1}),\,x_{k} - x_{k-1}\rangle + \frac{L}{2}\|x_{k}-x_{k-1}\|^2\notag\\
&= f(x_{k-1}) - \gamma_{k}\langle\nabla f(x_{k-1}),\,G(y_{k};\xi_{k})\rangle + \frac{L\gamma_{k}^2}{2}\|G(y_{k};\xi_{k})\|^2\notag\\
&= f(x_{k-1}) -\gamma_{k} \langle\nabla f(x_{k-1}),\,\nabla f(y_{k}) \rangle - \gamma_{k}\langle\nabla f(x_{k-1}),\,G(y_{k};\xi_{k}) - \nabla f(y_{k})\rangle\notag\\ 
&\qquad + \frac{L\gamma_{k}^2}{2}\|G(y_{k};\xi_{k})\|^2\notag\\
&=  f(x_{k-1}) -\gamma_{k} \|\nabla f(y_{k})\|^2 - \gamma_{k}\langle\nabla f(x_{k-1}) - \nabla f(y_{k}),\,\nabla f(y_{k}) \rangle \notag\\
&\qquad - \gamma_{k}\langle\nabla f(x_{k-1}),\,G(y_{k};\xi_{k}) - \nabla f(y_{k})\rangle + \frac{L\gamma_{k}^2}{2}\|G(y_{k};\xi_{k})-\nabla f(y_{k}) + \nabla f(y_{k} )\|^2\notag\\
&= f(x_{k-1}) -\gamma_{k}\left(1 - \frac{L\gamma_{k}}{2}\right) \|\nabla f(y_{k})\|^2 - \gamma_{k}\langle\nabla f(x_{k-1}) - \nabla f(y_{k}),\,\nabla f(y_{k}) \rangle \notag\\
&\qquad - \gamma_{k}\langle\nabla f(x_{k-1}) - L\gamma_{k}\nabla f(y_{k}),\,G(y_{k};\xi_{k}) - \nabla f(y_{k})\rangle + \frac{L\gamma_{k}^2}{2}\|G(y_{k};\xi_{k})-\nabla f(y_{k})\|^2\notag\\
&\leq f(x_{k-1}) -\gamma_{k}\left(1 - \frac{L\gamma_{k}}{2}\right) \|\nabla f(y_{k})\|^2 + L\gamma_{k}\|x_{k-1} - y_{k}\|\|\nabla f(y_{k})\|\notag\\
&\qquad - \gamma_{k}\langle\nabla f(x_{k-1}) - L\gamma_{k}\nabla f(y_{k}),\,G(y_{k};\xi_{k}) - \nabla f(y_{k})\rangle + 2M^2 L\gamma_{k}^2,
\end{align*}
where the last inequality is due to  \eqref{ass_Lipschitz:Ineq1} and \eqref{assump_bounded_gradient:Ineq}. Using \eqref{alg_nonconvex:xbar} we have from the preceding relation
\begin{align}
f(x_{k}) &\leq f(x_{k-1}) -\gamma_{k}\left(1 - \frac{L\gamma_{k}}{2}\right) \|\nabla f(y_{k})\|^2 + L(1-\alpha_{k})\gamma_{k}\|x_{k-1} - \xbar_{k-1}\|\|\nabla f(y_{k})\|\notag\\
&\qquad - \gamma_{k}\langle\nabla f(x_{k-1}) - L\gamma_{k}\nabla f(y_{k}),\,G(y_{k};\xi_{k}) - \nabla f(y_{k})\rangle + 2M^2 L\gamma_{k}^2\notag\\
& \leq f(x_{k-1}) -\gamma_{k}\left(1 - L\gamma_{k}\right) \|\nabla f(y_{k})\|^2 + \frac{L(1-\alpha_{k})^2}{2}\|x_{k-1} - \xbar_{k-1}\|^2\notag\\
&\qquad - \gamma_{k}\langle\nabla f(x_{k-1}) - L\gamma_{k}\nabla f(y_{k}),\,G(y_{k};\xi_{k}) - \nabla f(y_{k})\rangle  + 2M^2 L\gamma_{k}^2, \label{lem_nonconvex:Eq1}
\end{align}
where the last inequality we apply the relation $2ab\leq a^2 + b^2$ to the third term. Next, using \eqref{alg_nonconvex:y}--\eqref{alg_nonconvex:xbar} we have
\begin{align*}
\xbar_{k} - x_{k} = (1-\alpha_{k})(\xbar_{k-1}-x_{k-1}) + (\gamma_{k}-\beta_{k}) G(y_{k};\xi_{k}),    
\end{align*}
which dividing both sides by $\Gamma_{k}$, using \eqref{notation:Gamma} and $\alpha_{1} = 1$, and summing up both sides yields
\begin{align}
\xbar_{k} - x_{k} = \Gamma_{k}\sum_{t=1}^{k}\frac{\gamma_{t}-\beta_{t}}{\Gamma_{t}}G(y_{t};\xi_{t}). \label{lem_nonconvex:Eq1aa}
\end{align}
In addition, we also have
\begin{align*}
\sum_{k=1}^{K}\frac{\alpha_{k}}{\Gamma_{k}} = \frac{\alpha_1}{\Gamma_1} + \sum_{k=2}^K \frac{1}{\Gamma_k} \left( 1 - \frac{\Gamma_k}{\Gamma_{k-1}}  \right) = \frac{1}{\Gamma_1} + \sum_{k=2}^K \left( \frac{1}{\Gamma_k} - \frac{1}{\Gamma_{k-1}}  \right) = \frac{1}{\Gamma_{K}}\cdot
\end{align*}
Thus, by using the Jensen's inequality for $\|\cdot\|^2$ we have from the preceding two equations
\begin{align}
\|\xbar_{k} - x_{k}\|^2 &=  \left\|\Gamma_{k}\sum_{t=1}^{k}\frac{\gamma_{t}-\beta_{t}}{\Gamma_{t}}G(y_{t};\xi_{t})\right\|^2 = \left\|\Gamma_{k}\sum_{t=1}^{k}\frac{\alpha_{t}}{\Gamma_{t}}\frac{\gamma_{t}-\beta_{t}}{\alpha_{t}}G(y_{t};\xi_{t})\right\|^2\notag\\
&\leq \Gamma_{k}\sum_{t=1}^{k}\frac{\alpha_{t}}{\Gamma_{t}}\left\|\frac{\gamma_{t}-\beta_{t}}{\alpha_{t}}G(y_{t};\xi_{t})\right\|^2\leq M^2\Gamma_{k}\sum_{t=1}^{k}\frac{(\gamma_{t}-\beta_{t})^2}{\Gamma_{t}\alpha_{t}},   \label{lem_nonconvex:Eq1a}
\end{align}
where the last inequality is due to \eqref{assump_bounded_gradient:Ineq}. Substituing the preceding relation into \eqref{lem_nonconvex:Eq1} and since $(1-\alpha_{k})^2\Gamma_{k-1}\leq \Gamma_{k}$ we have
\begin{align}
f(x_{k}) & \leq f(x_{k-1}) -\gamma_{k}\left(1 - L\gamma_{k}\right) \|\nabla f(y_{k})\|^2 + \frac{M^2 L\Gamma_{k}}{2}\sum_{t=1}^{k}\frac{(\gamma_{t}-\beta_{t})^2}{\Gamma_{t}\alpha_{t}}\notag\\
&\qquad - \gamma_{k}\langle\nabla f(x_{k-1}) - L\gamma_{k}\nabla f(y_{k}),\,G(y_{k};\xi_{k}) - \nabla f(y_{k})\rangle  + 2M^2 L\gamma_{k}^2\notag\\
&\leq f(x_{k-1}) -\gamma_{k}\left(1 - L\gamma_{k}\right) \|\nabla f(y_{k})\|^2 + \frac{M^2 L\Gamma_{k}}{2}\sum_{t=1}^{k}\frac{(\gamma_{t}-\beta_{t})^2}{\Gamma_{t}\alpha_{t}}\notag\\
&\qquad - \gamma_{k}\langle\nabla f(x_{k-1}),\,G(y_{k};\xi_{k}) - \nabla f(y_{k})\rangle + 4M^2 L\gamma_{k}^2, \label{lem_nonconvex:Eq2}
\end{align}
where the last inequality is due to \eqref{assump_bounded_gradient:Ineq}. We next analyze the inner product on the right-hand side of \eqref{lem_nonconvex:Eq2}. Indeed, by Assumptions \ref{ass:Lipschitz} and \ref{assump:bounded_gradient} we have
\begin{align}
&- \gamma_{k}\langle \nabla f(x_{k-1}) ,\,G(y_{k};\xi_{k}) - \nabla f(y_{k})\rangle\notag\\ 
&= - \gamma_{k}\langle  \nabla f(x_{k-\tau(\gamma_{k})}),\,G(y_{k};\xi_{k}) - \nabla f(y_{k})\rangle\notag\\ 
&\qquad - \gamma_{k}\langle \nabla f(x_{k-1}) - \nabla f(x_{k-\tau(\gamma_{k})}),\,G(y_{k};\xi_{k}) - \nabla f(y_{k})\rangle\notag\\
&= - \gamma_{k}\langle  \nabla f(x_{k-\tau(\gamma_{k})}),\,G(y_{k-\tau(\gamma_{k})};\xi_{k}) - \nabla f(y_{k-\tau(\gamma_{k})})\rangle\notag\\ 
&\qquad - \gamma_{k}\langle  \nabla f(x_{k-\tau(\gamma_{k})}),\,G(y_{k});\xi_{k}) - G(y_{k-\tau(\gamma_{k})};\xi_{k})\rangle\notag\\
&\qquad - \gamma_{k}\langle  \nabla f(x_{k-\tau(\gamma_{k})}),\,\nabla f(y_{k-\tau(\gamma_{k})})-\nabla f(y_{k})\rangle\notag\\
&\qquad - \gamma_{k}\langle \nabla f(x_{k-1}) - \nabla f(x_{k-\tau(\gamma_{k})}),\,G(y_{k};\xi_{k}) - \nabla f(y_{k})\rangle\notag\\
&\leq - \gamma_{k}\langle  \nabla f(x_{k-\tau(\gamma_{k})}),\,G(y_{k-\tau(\gamma_{k})};\xi_{k}) - \nabla f(y_{k-\tau(\gamma_{k})})\rangle + 2ML\gamma_{k}\|y_{k}-y_{k-\tau(\gamma_{k})}\|\notag\\
&\qquad + 2LM\gamma_{k}\|x_{k-1}-x_{k-\tau(\gamma_{k})}\|.
 \label{lem_nonconvex:Eq2a}
\end{align}
First, we denote by $\Fcal_{k}$ the filtration containing all the history generated by the algorithm up to time $k$. Using Eq.\ \eqref{appendix:cor:mixing:ineq} and Assumption \ref{assump:ergodicity} we consider
\begin{align*}
&\Eset[- \gamma_{k}\langle  \nabla f(x_{k-\tau(\gamma_{k})}),\,G(y_{k-\tau(\gamma_{k})};\xi_{k}) - \nabla f(y_{k-\tau(\gamma_{k})})\rangle\,|\,\Fcal_{k-\tau(\gamma_{k})}]\notag\\ 
&= -\gamma_{k}   \langle  \nabla f(x_{k-\tau(\gamma_{k})}),\,\Eset[G(y_{k-\tau(\gamma_{k})};\xi_{k}) - \nabla f(y_{k-\tau(\gamma_{k})})\,|\,\Fcal_{k-\tau(\gamma_{k})}]\rangle \notag\\
& \leq M\gamma_{k}\left|\Eset[G(y_{k-\tau(\gamma_{k})};\xi_{k}) - \nabla f(y_{k-\tau(\gamma_{k})})\,|\,\Fcal_{k-\tau(\gamma_{k})}]\right| \leq M\gamma_{k}^2. 
\end{align*}
Second, by Eq.\ \eqref{alg_nonconvex:xbar} we have
\begin{align*}
y_{k+1} - y_{k} &= y_{k+1} - \xbar_{k+1} + \xbar_{k+1} - \xbar_{k} +\xbar_{k} - y_{k}\notag\\ 
&= \beta_{k+1}G(y_{k+1};\xi_{k+1}) - \beta_{k}G(y_{k};\xi_{k})+ \xbar_{k+1} - \xbar_{k},
\end{align*}
which by Assumption \ref{assump:bounded_gradient} and since $\beta_{k+1}\leq \beta_{k}$  implies that
\begin{align}
\|y_{k+1}-y_{k}\| \leq 2M\beta_{k} +    \|\xbar_{k+1} - \xbar_{k}\|. \label{lem_nonconvex:Eq2b} 
\end{align}
Using \eqref{alg_nonconvex:y} and \eqref{alg_nonconvex:xbar} we have
\begin{align*}
\xbar_{k+1} -\xbar_{k} = \alpha_{k+1}(x_{k}-\xbar_{k}) -\beta_{k+1}G(y_{k+1};\xi_{k+1}),
\end{align*}
Using \eqref{lem_nonconvex:Eq1aa} and the triangle inequality we have
\begin{align*}
\|\xbar_{k} - x_{k}\| &=  \left\|\Gamma_{k}\sum_{t=1}^{k}\frac{\gamma_{t}-\beta_{t}}{\Gamma_{t}}G(y_{t};\xi_{t})\right\| \leq M\Gamma_{k}\sum_{t=1}^{k}\frac{\gamma_{t}-\beta_{t}}{\Gamma_{t}}.
\end{align*}
Thus, by using $\alpha_{k+1} \leq \alpha_{k}$ and Assumption \ref{assump:bounded_gradient} we obtain from the two preceding equations
\begin{align*}
\|\xbar_{k+1} -\xbar_{k}\| \leq     M\alpha_{k}\Gamma_{k}\sum_{t=1}^{k}\frac{\gamma_{t}-\beta_{t}}{\Gamma_{t}} + M\beta_{k}.
\end{align*}
Substituting the preceding relation into \eqref{lem_nonconvex:Eq2b} yields
\begin{align*}
\|y_{k+1}-y_{k}\| \leq  3M\beta_{k} + M\alpha_{k}\Gamma_{k}\sum_{t=1}^{k}\frac{\gamma_{t}-\beta_{t}}{\Gamma_{t}},  
\end{align*}
which since $\beta_{k}$ is nonincreasing and $\beta_{k}\leq \gamma_{k}$, gives 
\begin{align*}
&2LM\gamma_{k}\|y_{k}-y_{k-\tau(\gamma_{k})}\|\leq 2LM\gamma_{k}\sum_{t=k-\tau(\gamma_{k})}^{k-1}\|y_{t+1}-y_{t}\|\notag\\ 
&\leq 6LM^2\gamma_{k}\sum_{t=k-\tau(\gamma_{k})}^{k-1}\beta_{t} + 2LM^3\gamma_{k}\sum_{t=k-\tau(\gamma_{k})}^{k-1}\alpha_{t}\Gamma_{t}\sum_{u=1}^{t}\frac{\gamma_{u}-\beta_{u}}{\Gamma_{u}}\notag\\
&\leq 6LM^2\tau(\gamma_{k})\gamma_{k}\beta_{k-\tau(\gamma_{k})} + 2LM^3 \gamma_{k}\sum_{t=k-\tau(\gamma_{k})}^{k-1}\alpha_{t}\Gamma_{t}\sum_{t=1}^{k}\frac{\gamma_{t}-\beta_{t}}{\Gamma_{t}}\notag\\
&\leq 6LM^2\tau(\gamma_{k})\gamma_{k-\tau(\gamma_{k})}\gamma_{k} + 2LM^3 \gamma_{k}\sum_{t=k-\tau(\gamma_{k})}^{k-1}\alpha_{t}\Gamma_{t}\sum_{t=1}^{k}\frac{\gamma_{t}-\beta_{t}}{\Gamma_{t}}\cdot
\end{align*}
Third, by \eqref{alg_nonconvex:x} we have 
\begin{align*}
&2LM\gamma_{k}\|x_{k-1}-x_{k-\tau(\gamma_{k})}\|\leq 2LM\gamma_{k}\sum_{t=k-\tau(\gamma_{k})+1}^{k-1}\|x_{t}-x_{t-1}\|\notag\\
& = 2LM\gamma_{k}  \sum_{t=k-\tau(\gamma_{k})+1}^{k-1}\|\gamma_{t}G(y_{t};\xi_{t})\|\leq 2LM^2\tau(\gamma_{k})\gamma_{k-\tau(\gamma_{k})}\gamma_{k}.   
\end{align*}
Taking the expectation on both sides of \eqref{lem_nonconvex:Eq2a} and using the preceding three relations we obtain
\begin{align*}
& \Eset\left[- \gamma_{k}\langle \nabla f(x_{k-1}) ,\,G(y_{k};\xi_{k}) - \nabla f(y_{k})\rangle\right]\notag\\
&\quad \leq M\gamma_{k}^2 + 8LM^2\tau(\gamma_{k})\gamma_{k-\tau(\gamma_{k})}\gamma_{k} +  2LM^3 \gamma_{k}\sum_{t=k-\tau(\gamma_{k})}^{k-1}\alpha_{t}\Gamma_{t}\sum_{t=1}^{k}\frac{\gamma_{t}-\beta_{t}}{\Gamma_{t}}.
\end{align*}
Taking the expectation on both sides of \eqref{lem_nonconvex:Eq2} and using the equation above give \eqref{lem_nonconvex:Ineq}, i.e., 
\begin{align*}
\Eset\left[f(x_{k})\right]
&\leq \Eset\left[f(x_{k-1})\right] -\gamma_{k}\left(1 - L\gamma_{k}\right) \Eset\left[\|\nabla f(y_{k})\|^2\right] + \frac{M^2 L\Gamma_{k}}{2}\sum_{t=1}^{k}\frac{(\gamma_{t}-\beta_{t})^2}{\Gamma_{t}\alpha_{t}}\notag\\
&\qquad   + (4ML + 1)M \gamma_{k}^2     + 8LM^2\tau(\gamma_{k})\gamma_{k-\tau(\gamma_{k})}\gamma_{k} +  2LM^3 \gamma_{k}\sum_{t=k-\tau(\gamma_{k})}^{k-1}\alpha_{t}\Gamma_{t}\sum_{t=1}^{k}\frac{\gamma_{t}-\beta_{t}}{\Gamma_{t}}\cdot
\end{align*}
\end{proof}


\subsection{Proofs of Section \ref{sec:convex}}
The analysis of Theorems \ref{thm:convex} and \ref{thm:sc} are established based on the following two key lemmas. The proof of the first lemma is adopted from the results studied in \cite{Lan2012}. We restate here with some minor modification for the purpose of our analysis.

\begin{lem}\label{lem:f_xbar}
Let $\alpha_{k}$ and $\gamma_{k}$ satisfy \eqref{condition:stepsizes} and
\begin{align}
\frac{\alpha_{k}}{\gamma_{k}\Gamma_{k}}\leq\frac{\alpha_{k-1}(1+\mu\gamma_{k-1})}{\gamma_{k-1}\Gamma_{k-1}},    \label{lem_f_xbar:stepsizes}
\end{align}
where $\Gamma_{k}$ is defined in \eqref{notation:Gamma}. Then $\{\xbar_{k}\}$ generated by Algorithm \ref{alg:ASGD_convex} satisfies for all $k\geq 1$
\begin{align}
 f(\xbar_{k}) - f(x^*) &\leq \Gamma_{k}\frac{\gamma_{0}f(\xbar_{0}) + \alpha_{0}(1+\mu\gamma_{0})D}{\gamma_{0}\Gamma_{0}}  + \Gamma_{k}\sum_{t=1}^{k} \frac{4M^2\gamma_{t}\alpha_{t}}{\Gamma_{t}(1+\mu\gamma_{t}-L\gamma_{t}\alpha_{t})}\notag\\
&\quad+ \Gamma_{k}\sum_{t=1}^{k}\frac{\alpha_{t}}{\Gamma_{t}}\big\langle G(y_{t};\xi_{t})-\nabla f(y_{t})\,,\, z - \tilde{x}_{t-1}\big\rangle,\label{lem_f_xbar:Ineq}
\end{align}
where $\tilde{x}_{k-1}$ is defined as
\begin{align}
\tilde{x}_{k-1} = \frac{1}{1+\mu\gamma_{k}}x_{k-1} + \frac{\mu\gamma_{k}}{1+\mu\gamma_{k}}y_{k}.\label{apx_notation:xtilde}    
\end{align}
\end{lem}

\begin{proof}
Using the convexity of $f$, i.e., $$f(y) \geq f(x) + \langle \nabla f(x), y - x \rangle,\quad  \forall x, y \in \mathbb{R}^d$$ and \eqref{alg:xbar} we have for all $z\in\Xcal$
\begin{align*}
f(z) + \langle \nabla f(z),\xbar_{k} - z \rangle  &= f(z) + \langle \nabla f(z),\alpha_{k}x_{k} + (1-\alpha_{k})\xbar_{k-1} - z \rangle\\
&= (1-\alpha_{k})\big[f(z) + \langle \nabla f(z),\xbar_{k-1} - z \rangle\big] + \alpha_{k}\big[f(z) + \langle \nabla f(z),x_{k} - z \rangle\big]\\
&\leq (1-\alpha_{k})f(\xbar_{k-1}) + \alpha_{k}\big[f(z) + \langle \nabla f(z),x_{k} - z \rangle\big].
\end{align*}
By the preceding relation and  \eqref{ass_Lipschitz:Ineq1} we have for all $z\in\Xcal$
\begin{align*}
f(\xbar_{k}) &\leq f(z) + \langle \nabla f(z),\xbar_{k} - z \rangle + \frac{L}{2}\|\xbar_{k} - z\|^2 \\
&\leq (1-\alpha_{k})f(\xbar_{k-1}) + \alpha_{k}\big[f(z) + \langle \nabla f(z),x_{k} - z \rangle\big]+ \frac{L}{2}\|\xbar_{k} - z\|^2 ,
\end{align*}
which by letting $z = y_{k}$ we obtain
\begin{align}
f(\xbar_{k}) 
&\leq (1-\alpha_{k})f(\xbar_{k-1}) + \alpha_{k}\big[f(y_{k}) + \langle \nabla f(y_{k}),x_{k} - y_{k} \rangle\big]+ \frac{L}{2}\|\xbar_{k} - y_{k}\|^2.\label{apx_lem:Eq1}
\end{align}
By the update of $x_{k}$ in  \eqref{alg:x} and Lemma 3.5 in \cite{Lan2019}  we have for all $z\in\Xcal$
\begin{align*}
&\gamma_{k}\left[\big\langle G(y_{k};\xi_{k})\,,\, x_{k} - y_{k}\big\rangle + \mu V(y_{k},x_{k})\right] + V(x_{k-1},x_{k})\\
&\qquad \leq \gamma_{k}\left[\big\langle G(y_{k};\xi_{k})\,,\, z - y_{k}\big\rangle + \mu V(y_{k},z)\right] +  V(x_{k-1},z) - (1+\mu\gamma_{k}) V(x_{k},z), 
\end{align*}
which implies that
\begin{align*}
\big\langle \nabla f(y_{k})\,,\, x_{k} - y_{k}\big\rangle &\leq  \mu V(y_{k},z) +  \frac{1}{\gamma_{k}}V(x_{k-1},z) - \frac{1+\mu\gamma_{k}}{\gamma_{k}} V(x_{k},z)- \mu V(y_{k},x_{k})- \frac{1}{\gamma_{k}}V(x_{k-1},x_{k})\notag\\
&\qquad + \big\langle G(y_{k};\xi_{k})\,,\, z - y_{k}\big\rangle  - \big\langle G(y_{k};\xi_{k})-\nabla f(y_{k})\,,\, x_{k} - y_{k}\big\rangle\notag\\
&=  \mu V(y_{k},z) +  \frac{1}{\gamma_{k}}V(x_{k-1},z) - \frac{1+\mu\gamma_{k}}{\gamma_{k}} V(x_{k},z)- \mu V(y_{k},x_{k})- \frac{1}{\gamma_{k}}V(x_{k-1},x_{k})
\notag\\
&\qquad + \big\langle \nabla f(y_k)\,,\,z - y_{k}\big\rangle +  \big\langle G(y_{k};\xi_{k})-\nabla f(y_{k})\,,\, z - x_{k}\big\rangle.
\end{align*}
Substituting the preceding equation into Eq.\ \eqref{apx_lem:Eq1} yields
\begin{align}
f(\xbar_{k}) 
&\leq (1-\alpha_{k})f(\xbar_{k-1}) + \alpha_{k}\big[f(y_{k}) + \big\langle \nabla f(y_k)\,,\,z - y_{k}\big\rangle + \mu V(y_{k},z) \big]+ \frac{L}{2}\|\xbar_{k} - y_{k}\|^2 \notag\\
&\qquad + \frac{\alpha_{k}}{\gamma_{k}}\Big[ V(x_{k-1},z) - (1+\mu\gamma_{k}) V(x_{k},z) - V(x_{k-1},x_{k}) - \mu\gamma_{k} V(y_{k},x_{k})\Big]
\notag\\
&\qquad +  \alpha_{k}\big\langle G(y_{k};\xi_{k})-\nabla f(y_{k})\,,\, z - x_{k}\big\rangle.\label{lem_f_xbar:Eq2}
\end{align}
We denote by $\tilde{x}_{k-1}$
\begin{align}
\tilde{x}_{k-1} = \frac{1}{1+\mu\gamma_{k}}x_{k-1} + \frac{\mu\gamma_{k}}{1+\mu\gamma_{k}}y_{k}.\label{apx_notation:xtilde}    
\end{align}
And note that 
\begin{align*}
\frac{1}{1+\mu\gamma_{k}} = \frac{\beta_{k}(1-\alpha_{k})}{\alpha_{k}(1-\beta_{k})}\qquad \text{and}\qquad \frac{\mu\gamma_{k}}{1+\mu\gamma_{k}} = \frac{\alpha_{k}-\beta_{k}}{\alpha_{k}(1-\beta_{k})}     
\end{align*}
Thus, using Eqs.\ \eqref{alg:y} and \eqref{alg:xbar}, and the preceding relations we then have
\begin{align}
\xbar_{k} - y_{k} &= \alpha_{k}x_{k} + \frac{1-\alpha_{k}}{1-\beta_{k}}   \left(y_{k}-\beta_{k}x_{k-1}\right) -y_{k} = \alpha_{k}\left[x_{k} - \frac{\beta_{k}(1-\alpha_{k})}{\alpha_{k}(1-\beta_{k})}x_{k-1} - \frac{\alpha_{k}-\beta_{k}}{\alpha_{k}(1-\beta_{k})}y_{k}\right] \notag\\ &= \alpha_{k}(x_{k}-\tilde{x}_{k-1}).\label{lem_f_xbar:Eq2a}
\end{align}
On the other hand, using the strong convexity of $V$ we have
\begin{align}
V(x_{k-1},x_{k}) + \mu\gamma_{k}V(y_{k},x_{k}) &\geq \frac{1}{2}\|x_{k-1}-x_{k}\|^2 + \frac{\mu\gamma_{k}}{2}\|x_{k}-y_{k}\|^2\notag\\
&\geq \frac{1+\mu\gamma_{k}}{2} \left\|\frac{1}{1+\mu\gamma_{k}}\big(x_{k}-x_{k-1}\big) + \frac{\mu\gamma_{k}}{1+\mu\gamma_{k}}\big(x_{k}-y_{k}\big)\right\|^2\notag\\
&= \frac{1+\mu\gamma_{k}}{2}\left\|x_{k}-\frac{1}{1+\mu\gamma_{k}}x_{k-1} - \frac{\mu\gamma_{k}}{1+\mu\gamma_{k}}y_{k}\right\|^2\notag\\ 
&= \frac{1+\mu\gamma_{k}}{2\alpha_{k}^2}\left\|\alpha_{k}(x_{k}-\tilde{x}_{k-1})\right\|^2.\label{lem_f_xbar:Eq2b}
\end{align}
We denote by $\Delta_k = G(y_{k};\xi_{k})-\nabla f(y_{k})$. Then we consider
\begin{align}
\big\langle G(y_{k};\xi_{k})-\nabla f(y_{k})\,,\, z - x_{k}\big\rangle &= \big\langle G(y_{k};\xi_{k})-\nabla f(y_{k})\,,\, \tilde{x}_{k-1} - x_{k}\big\rangle + \big\langle G(y_{k};\xi_{k})-\nabla f(y_{k})\,,\, z - \tilde{x}_{k-1}\big\rangle\notag\\
&\leq \|\Delta_{k}\|\|x_{k}-\tilde{x}_{k-1}\| + \big\langle \Delta_{k}\,,\, z - \tilde{x}_{k-1}\big\rangle.\label{lem_f_xbar:Eq2c}    
\end{align}
Substituting Eqs.\ \eqref{lem_f_xbar:Eq2a}--\eqref{lem_f_xbar:Eq2c} into Eq.\ \eqref{lem_f_xbar:Eq2} yields \begin{align}
f(\xbar_{k}) 
&\leq (1-\alpha_{k})f(\xbar_{k-1}) + \alpha_{k}\big[f(y_{k}) + \big\langle \nabla f(y_k)\,,\,z - y_{k}\big\rangle + \mu V(y_{k},z) \big]\notag\\
&\qquad + \frac{\alpha_{k}}{\gamma_{k}}\Big[ V(x_{k-1},z) - (1+\mu\gamma_{k}) V(x_{k},z)\Big] + \alpha_{k}\big\langle G(y_{k};\xi_{k})-\nabla f(y_{k})\,,\, z - \tilde{x}_{k-1}\big\rangle
\notag\\
&\qquad -\left(\frac{1+\mu\gamma_{k}}{2\gamma_{k}\alpha_{k}} - \frac{L}{2}\right)\|\alpha_{k}(x_{k}-\tilde{x}_{k-1})\|^2 + \|\Delta_{k}\|\|\alpha_{k}(x_{k}-\tilde{x}_{k-1})\|\notag\\
&\leq    (1-\alpha_{k})f(\xbar_{k-1}) + \alpha_{k}\big[f(y_{k}) + \big\langle \nabla f(y_k)\,,\,z - y_{k}\big\rangle + \mu V(y_{k},z) \big]\notag\\
&\qquad + \frac{\alpha_{k}}{\gamma_{k}}\Big[ V(x_{k-1},z) - (1+\mu\gamma_{k}) V(x_{k},z)\Big] + \alpha_{k}\big\langle G(y_{k};\xi_{k})-\nabla f(y_{k})\,,\, z - \tilde{x}_{k-1}\big\rangle \notag\\ &\qquad + \frac{\gamma_{k}\alpha_{k}\|\Delta_{k}\|^2}{1+\mu\gamma_{k}-L\gamma_{k}\alpha_{k}}\cdot\ \label{lem_f_xbar:Eq3}    
\end{align}
Diving both sides of Eq.\ \eqref{lem_f_xbar:Eq3} by $\Gamma_{k}$ and using \eqref{notation:Gamma}  we have 
\begin{align*}
\frac{1}{\Gamma_{k}}f(\xbar_{k}) &\leq \frac{1-\alpha_{k}}{\Gamma_{k}}f(\xbar_{k-1}) + \frac{\alpha_{k}}{\Gamma_{k}}\big[f(y_{k}) + \big\langle \nabla f(y_k)\,,\,z - y_{k}\big\rangle + \mu V(y_{k},z) \big]+ \frac{\gamma_{k}\alpha_{k}\|\Delta_{k}\|^2}{\Gamma_{k}(1+\mu\gamma_{k}-L\gamma_{k}\alpha_{k})}\notag\\
&\qquad + \frac{\alpha_{k}}{\Gamma_{k}\gamma_{k}}\Big[ V(x_{k-1},z) - (1+\mu\gamma_{k}) V(x_{k},z)\Big] + \frac{\alpha_{k}}{\Gamma_{k}}\big\langle G(y_{k};\xi_{k})-\nabla f(y_{k})\,,\, z - \tilde{x}_{k-1}\big\rangle\notag\\
&\leq  \frac{1}{\Gamma_{k-1}}f(\xbar_{k-1}) + \frac{\alpha_{k}}{\Gamma_{k}}f(z) + \frac{\gamma_{k}\alpha_{k}\|\Delta_{k}\|^2}{\Gamma_{k}(1+\mu\gamma_{k}-L\gamma_{k}\alpha_{k})}\notag\\
&\qquad + \frac{\alpha_{k}}{\Gamma_{k}\gamma_{k}}\Big[ V(x_{k-1},z) - (1+\mu\gamma_{k}) V(x_{k},z)\Big] + \frac{\alpha_{k}}{\Gamma_{k}}\big\langle G(y_{k};\xi_{k})-\nabla f(y_{k})\,,\, z - \tilde{x}_{k-1}\big\rangle,
\end{align*}
where the last inequality due to the convexity of $f$.  Summing up both sides of the preceding relation over $k$ from $1$ to $K$ yields 
\begin{align*}
f(\xbar_{K}) &\leq \frac{\Gamma_{K}}{\Gamma_{0}}f(\xbar_{0}) + \Gamma_{K}\sum_{k=1}^{K}\frac{\alpha_{k}}{\Gamma_{k}}f(z) + \Gamma_{K}\sum_{k=1}^{K} \frac{\gamma_{k}\alpha_{k}\|\Delta_{k}\|^2}{\Gamma_{k}(1+\mu\gamma_{k}-L\gamma_{k}\alpha_{k})}\notag\\
&\qquad + \Gamma_{K}\sum_{k=1}^{K}\frac{\alpha_{k}}{\Gamma_{k}\gamma_{k}}\Big[ V(x_{k-1},z) - (1+\mu\gamma_{k}) V(x_{k},z)\Big] \notag\\
&\qquad+ \Gamma_{K}\sum_{k=1}^{K}\frac{\alpha_{k}}{\Gamma_{k}}\big\langle G(y_{k};\xi_{k})-\nabla f(y_{k})\,,\, z - \tilde{x}_{k-1}\big\rangle\\
&\leq \frac{\Gamma_{K}}{\Gamma_{0}} f(\xbar_{0}) + f(z) + \Gamma_{K}\sum_{k=1}^{K} \frac{\gamma_{k}\alpha_{k}\|\Delta_{k}\|^2}{\Gamma_{k}(1+\mu\gamma_{k}-L\gamma_{k}\alpha_{k})}\notag\\
&\qquad + \Gamma_{K}\left[ \frac{\alpha_{0}(1+\mu\gamma_{0})}{\gamma_{0}\Gamma_{0}} V(x_{0},z) - \frac{\alpha_{K}(1+\mu\gamma_{K})}{\Gamma_{K}\gamma_{K}} V(x_{K},z)\right] \notag\\
&\qquad + \Gamma_{K}\sum_{k=1}^{K}\frac{\alpha_{k}}{\Gamma_{k}}\big\langle G(y_{k};\xi_{k})-\nabla f(y_{k})\,,\, z - \tilde{x}_{k-1}\big\rangle,
\end{align*}
where the second inequality is due to \eqref{lem_f_xbar:stepsizes}, $\alpha_1 = 1$, and the definition of $\Gamma_{k}$ to have
\begin{align}
\sum_{k=1}^{K}\frac{\alpha_{k}}{\Gamma_{k}} = \frac{\alpha_1}{\Gamma_1} + \sum_{k=2}^K \frac{1}{\Gamma_k} \left( 1 - \frac{\Gamma_k}{\Gamma_{k-1}}  \right) = \frac{1}{\Gamma_1} + \sum_{k=2}^K \left( \frac{1}{\Gamma_k} - \frac{1}{\Gamma_{k-1}}  \right) = \frac{1}{\Gamma_{K}}.\label{cond:alpha_Gamma}
\end{align}
Thus, by letting $z = x^*$ in the preceding equation and since $\|\Delta_{k}\|^2 \leq 4M^2$ we obtain \eqref{lem_f_xbar:Ineq}.
\end{proof}

\subsubsection{Proof of Theorem \ref{thm:convex}}
To prove theorem \ref{thm:convex}, we first consider the following lemma, where  handle the inner product on the right-hand side of \eqref{lem_f_xbar:Ineq} by using the geometric mixing time, similar to Lemma \ref{lem:nonconvex}. Recall that since $\mu = 0$, we have $\beta_{k} = \alpha_{k}$ and $y_{k} = \xbar_{k}$. Thus, the updates in Algorithm \ref{alg:ASGD_convex} become
\begin{align}
    \xbar_{k} &= (1-\alpha_{k})\xbar_{k-1} + \alpha_{k}x_{k-1}\label{apx_alg:xbar}\\
    x_{k} &= \arg\min_{x\in\Xcal}\Big\{\gamma_{k}\langle G(\xbar_{k};\xi_{k})\,,\,x- \xbar_{k}\rangle + \frac{1}{2}\|x-x_{k-1}\|^2\Big\}.\label{apx_alg:x}
\end{align}

\begin{lem}\label{apx_lem_convex:mixing_grad}
Let the sequences $\{x_{k},y_{k}\}$ be generated by Algorithm \ref{alg:ASGD_convex}. Then 
\begin{align}
\Eset[\langle G(\xbar_{k};\xi_{k})-\nabla f(\xbar_{k}), z-x_{k-1} \rangle] &\leq 2D\gamma_{k} + 2(4D^2L + M^2)\tau(\gamma_{k})\gamma_{k-\tau(\gamma_{k})} .    \label{apx_lem_convex_mixing_grad:Ineq}
\end{align}
\end{lem}

\begin{proof}
First, by the optimality condition of \eqref{apx_alg:x} we have
\begin{align*}
\langle\gamma_{k}G(\xbar_{k};\xi_{k}) + x_{k}-x_{k-1}\,,\,x_{k-1} - x_{k}  \rangle \geq 0,    
\end{align*}
which by rearranging the equation and using \eqref{notation:DM} we have
\begin{align*}
\|x_{k}-x_{k-1}\|^2 \leq \langle\gamma_{k}G(\xbar_{k};\xi_{k})\,,\,x_{k-1} - x_{k}  \rangle \leq M\gamma_{k}\|x_{k}-x_{k-1}\|,    
\end{align*}
Dividing both sides of the equation above by $x_{k}-x_{k-1}$ gives
\begin{align} 
\|x_{k}-x_{k-1}\|\leq M\gamma_{k}.\label{apx_lem_convex_mixing_grad:Eq1a}
\end{align}
Since $\mu=0$, $\tilde{x}_{k-1} = x_{k-1}$. Next, we consider
\begin{align}
\langle G(\xbar_{k};\xi_{k})-\nabla f(\xbar_{k}), z-x_{k-1} \rangle &= \langle G(\xbar_{k};\xi_{k}) - \nabla f(\xbar_{k})\,,\, z - x_{k-\tau(\gamma_{k})}  \rangle \notag\\
&\qquad + \langle G(\xbar_{k};\xi_{k}) - \nabla f(\xbar_{k})\,,\, x_{k-\tau(\gamma_{k})} - x_{k-1}  \rangle.\label{apx_lem_convex_mixing_grad:Eq1}    
\end{align}
We now provide upper bounds for each term on the right-hand side of Eq.\ \eqref{apx_lem_convex_mixing_grad:Eq1}. First, using \eqref{notation:DM} consider the first term on the right-hand side of Eq.\ \eqref{apx_lem_convex_mixing_grad:Eq1} 
\begin{align*}
&\langle G(\xbar_{k};\xi_{k})-\nabla f(\xbar_{k}), z-x_{k-\tau(\gamma_{k})}\rangle\notag\\ 
&= \langle G(\xbar_{k};\xi_{k})-G(\xbar_{k-\tau(\gamma_{k})};\xi_{k}), z-x_{k-\tau(\gamma_{k})}\rangle + \langle G(\xbar_{k-\tau(\gamma_{k})};\xi_{k})-\nabla f(\xbar_{k-\tau(\gamma_{k})}), z-x_{k-\tau(\gamma_{k})}\rangle\notag\\
&\qquad + \langle \nabla f(\xbar_{k-\tau(\gamma_{k})})-\nabla f(\xbar_{k}), z-x_{k-\tau(\gamma_{k})}\rangle\notag\\ 
&\leq 2L\|\xbar_{k}-\xbar_{k-\tau(\gamma_{k})}\|\|z-x_{k-\tau(\gamma_{k})}\| + \langle G(\xbar_{k-\tau(\gamma_{k})};\xi_{k})-\nabla f(\xbar_{k-\tau(\gamma_{k})}), z-x_{k-\tau(\gamma_{k})}\rangle\notag\\
&\leq 4DL\sum_{t=k-\tau(\gamma_{k})}^{k-1}\|\xbar_{t+1}-\xbar_{t}\|  + \langle G(\xbar_{k-\tau(\gamma_{k})};\xi_{k})-\nabla f(\xbar_{k-\tau(\gamma_{k})}), z-x_{k-\tau(\gamma_{k})}\rangle\notag\\
&\leq 8D^2L\tau(\gamma_{k})\alpha_{k-\tau(\gamma_{k})} + \langle G(\xbar_{k-\tau(\gamma_{k})};\xi_{k})-\nabla f(\xbar_{k-\tau(\gamma_{k})}), z-x_{k-\tau(\gamma_{k})}\rangle, 
\end{align*}
where $\tau(\gamma_{k})$ be the mixing time of the underlying Markov chain associated with the step size $\gamma_{k}$, defined in \eqref{notation:tau}. We denote by $\Fcal_{k}$ the filtration containing all the history generated by the algorithm up to time $k$. Then, using \eqref{appendix:cor:mixing:ineq} we have 
\begin{align}
&\Eset[\langle G(\xbar_{k};\xi_{k})-\nabla f(\xbar_{k}), z-x_{k-\tau(\gamma_{k})}\,|\,\Fcal_{k-\tau(\gamma_{k})}\rangle]\notag\\  
&\leq 8D^2L\tau(\gamma_{k})\alpha_{k-\tau(\gamma_{k})} + \|z-x_{k-\tau(\gamma_{k})}\|\|\Eset[G(\xbar_{k};\xi_{k})-\nabla f(\xbar_{k})\,|\,\Fcal_{k-\tau(\gamma_{k})}\|\notag\\
&\leq 8D^2L\tau(\gamma_{k})\alpha_{k-\tau(\gamma_{k})} +  2D\gamma_{k}.\label{apx_lem_convex_mixing_grad:Eq1b}    
\end{align}
Second, using Eqs.\ \eqref{apx_lem_convex_mixing_grad:Eq1a} and \eqref{notation:DM} we consider the second term on the right-hand side of \eqref{apx_lem_convex_mixing_grad:Eq1}
\begin{align}
&\langle G(\xbar_{k};\xi_{k}) - \nabla f(\xbar_{k})\,,\, x_{k-\tau(\gamma_{k})} - x_{k-1}  \rangle\leq 2M\|x_{k-\tau(\gamma_{k})}-x_{k-1}\|\notag\\
& \leq 2M\sum_{t=k-\tau(\gamma_{k})}^{k-2}\|x_{t+1}-x_{t}\| \leq 2M^2\tau(\gamma_{k})\gamma_{k-\tau(\gamma_{k})}.
\label{apx_lem_convex_mixing_grad:Eq1c}
\end{align}
Taking the expectation on both sides of \eqref{apx_lem_convex_mixing_grad:Eq1} and using Eqs.\ \eqref{apx_lem_convex_mixing_grad:Eq1b}, \eqref{apx_lem_convex_mixing_grad:Eq1c}, and $\alpha_{k}\leq \gamma_{k}$  immediately gives Eq.\ \eqref{apx_lem_convex_mixing_grad:Ineq}.
\end{proof}

\subsubsection{Proof of Theorem \ref{thm:sc}}
Similar to Lemma \ref{apx_lem_convex:mixing_grad}, we start with the following lemma.
\begin{lem}\label{apx_lem_sc:mixing_grad}
Let the sequences $\{x_{k},y_{k}\}$ be generated by Algorithm \ref{alg:ASGD_convex} and $\tilde{x}$ is defined in \eqref{apx_notation:xtilde}. Then 
\begin{align}
\Eset[\langle G(y_{k};\xi_{k})-\nabla f(y_{k}), z-\tilde{x}_{k-1} \rangle] &\leq (2M^2 +4\mu MD+24\mu LD^2)\tau(\gamma_{k})\gamma_{k-\tau(\gamma_{k})}\notag\\ 
&\qquad + (2D+2M^2+8\mu MD)\gamma_{k}.    \label{apx_lem_sc_mixing_grad:Ineq}
\end{align}
\end{lem}

\begin{proof}
First, by the optimality condition of \eqref{alg:x} we have
\begin{align*}
\langle\gamma_{k}G(y_{k};\xi_{k}) + \mu\gamma_{k}(x_{k}-y_{k}) + x_{k}-x_{k-1}\,,\,x_{k-1} - x_{k}  \rangle \geq 0,    
\end{align*}
which by rearranging the equation and using \eqref{notation:DM} we have
\begin{align*}
\|x_{k}-x_{k-1}\|^2 \leq \langle\gamma_{k}G(y_{k};\xi_{k}) + \mu\gamma_{k}(x_{k}-y_{k})\,,\,x_{k-1} - x_{k}  \rangle \leq (M+2D\mu)\gamma_{k}\|x_{k}-x_{k-1}\|,    
\end{align*}
Dividing both sides of the equation above by $x_{k}-x_{k-1}$ gives
\begin{align} 
\|x_{k}-x_{k-1}\|\leq (M+2D\mu)\gamma_{k}.\label{apx_lem_sc_mixing_grad:Eq1a}
\end{align}
Second, we consider
\begin{align*}
y_{k+1} - y_{k} &= \xbar_{k} - \xbar_{k-1} -\beta_{k+1}(\xbar_{k} - x_{k}) +\beta_{k}(\xbar_{k-1} - x_{k-1})\notag\\
&= \alpha_{k}(x_{k}-\xbar_{k-1}) - -\beta_{k+1}(\xbar_{k} - x_{k}) +\beta_{k}(\xbar_{k-1} - x_{k-1}),
\end{align*}
which implies that
\begin{align}
\|y_{k+1}-y_{k}\| \leq 2D\alpha_{k} + 4D\beta_{k}\leq 6\mu D\gamma_{k},\label{apx_lem_sc_mixing_grad:Eq1b}
\end{align}
where we use the fact that $\beta_{k}\leq \alpha_{k} = \mu\gamma_{k}$. Third, we consider
\begin{align}
\langle G(y_{k};\xi_{k})-\nabla f(y_{k}), z-\tilde{x}_{k-1} \rangle &= \langle G(y_{k};\xi_{k}) - \nabla f(y_{k})\,,\, z - x_{k-\tau(\gamma_{k})}  \rangle \notag\\
&\qquad + \langle G(y_{k};\xi_{k}) - \nabla f(y_{k})\,,\, x_{k-\tau(\gamma_{k})} - x_{k}  \rangle\notag\\ 
&\qquad + \langle G(y_{k};\xi_{k}) - \nabla f(y_{k})\,,\, x_{k} - \tilde{x}_{k-1}\rangle.\label{apx_lem_sc_mixing_grad:Eq1c}    
\end{align}
Note that by \eqref{apx_notation:xtilde} we have
\begin{align*}
x_{k} - \tilde{x}_{k-1} &= x_{k}  - \frac{1}{1+\mu\gamma_{k}}x_{k-1}- \frac{\mu\gamma_{k}}{1+\mu\gamma_{k}}y_{k} = x_{k} - x_{k-1} +  \frac{\mu\gamma_{k}}{1+\mu\gamma_{k}}(x_{k-1}-y_{k}),
\end{align*}
which by substituing into Eq.\ \eqref{apx_lem_sc_mixing_grad:Eq1c} yields
\begin{align}
\langle G(y_{k};\xi_{k})-\nabla f(y_{k}), z-\tilde{x}_{k-1} \rangle &= \langle G(y_{k};\xi_{k}) - \nabla f(y_{k})\,,\, z - x_{k-\tau(\gamma_{k})}  \rangle \notag\\
&\qquad + \langle G(y_{k};\xi_{k}) - \nabla f(y_{k})\,,\, x_{k-\tau(\gamma_{k})} - x_{k}  \rangle\notag\\ 
&\qquad + \langle G(y_{k};\xi_{k}) - \nabla f(y_{k})\,,\, x_{k} - x_{k-1}\rangle \notag\\
&\qquad + \frac{\mu\gamma_{k}}{1+\mu\gamma_{k}}\langle G(y_{k};\xi_{k}) - \nabla f(y_{k})\,,\, x_{k} - y_{k}\rangle.\label{apx_lem_sc_mixing_grad:Eq1}    
\end{align}
Next, we analyze the right-hand side of Eq. \eqref{apx_lem_sc_mixing_grad:Eq1}. Using \eqref{notation:DM} and Assumption \ref{ass:Lipschitz}, consider the first term
\begin{align*}
&\langle G(y_{k};\xi_{k})-\nabla f(y_{k}), z-x_{k-\tau(\gamma_{k})}\rangle\notag\\ 
&= \langle G(y_{k};\xi_{k})-G(y_{k-\tau(\gamma_{k})};\xi_{k}), z-x_{k-\tau(\gamma_{k})}\rangle + \langle G(y_{k-\tau(\gamma_{k})};\xi_{k})-\nabla f(y_{k-\tau(\gamma_{k})}), z-x_{k-\tau(\gamma_{k})}\rangle\notag\\
&\qquad + \langle \nabla f(y_{k-\tau(\gamma_{k})})-\nabla f(y_{k}), z-x_{k-\tau(\gamma_{k})}\rangle\notag\\ 
&\leq 4DL\|y_{k}-y_{k-\tau(\gamma_{k})}\| + + \langle G(y_{k-\tau(\gamma_{k})};\xi_{k})-\nabla f(y_{k-\tau(\gamma_{k})}), z-x_{k-\tau(\gamma_{k})}\rangle\notag\\
&\leq 4DL \sum_{t=k-\tau(\gamma_{k})}^{k-1}\|y_{t+1}-y_{t}\| + \langle G(y_{k-\tau(\gamma_{k})};\xi_{k})-\nabla f(y_{k-\tau(\gamma_{k})}), z-x_{k-\tau(\gamma_{k})}\rangle\notag\\
&\leq 24\mu L D^2\sum_{t=k-\tau(\gamma_{k})}^{k-1}\gamma_{t}  + \langle G(y_{k-\tau(\gamma_{k})};\xi_{k})-\nabla f(y_{k-\tau(\gamma_{k})}), z-x_{k-\tau(\gamma_{k})}\rangle\notag\\
&\leq 24\mu L D^2\tau(\gamma_{k})\gamma_{k-\tau(\gamma_{k})} +  +  \langle G(y_{k-\tau(\gamma_{k})};\xi_{k})-\nabla f(y_{k-\tau(\gamma_{k})}), z-x_{k-\tau(\gamma_{k})}\rangle, 
\end{align*}
where the second last inequality is due to \eqref{apx_lem_sc_mixing_grad:Eq1b}. Taking the conditional expectation w.r.t $\Fcal_{k}$ and using \eqref{appendix:cor:mixing:ineq} yield
\begin{align}
&\Eset[\langle G(y_{k};\xi_{k})-\nabla f(y_{k}), z-x_{k-\tau(\gamma_{k})}\rangle\,|\,\Fcal_{k-\tau(\gamma_{k})}\rangle] \leq  24\mu L D^2\tau(\gamma_{k})\gamma_{k-\tau(\gamma_{k})} + 2D\gamma_{k}.\label{apx_lem_sc_mixing_grad:Eq1d}    
\end{align}
Second, using Eqs.\ \eqref{apx_lem_sc_mixing_grad:Eq1a} and \eqref{notation:DM} we consider the second and third terms on the right-hand side of \eqref{apx_lem_sc_mixing_grad:Eq1}
\begin{align}
&\langle G(y_{k};\xi_{k}) - \nabla f(y_{k})\,,\, x_{k-\tau(\gamma_{k})} - x_{k}  \rangle + \langle G(y_{k};\xi_{k}) - \nabla f(y_{k})\,,\, x_{k} - x_{k-1}\rangle\notag\\
& \leq 2M\|x_{k-\tau(\gamma_{k})}-x_{k}\| + 2M \|x_{k}-x_{k-1}\|\leq 2M\sum_{t=k+1-\tau(\gamma_{k})}^{k}\|x_{t}-x_{t-1}\| + 2M\|x_{k}-x_{k-1}\|\notag\\
&\stackrel{\eqref{apx_lem_sc_mixing_grad:Eq1a}}{\leq} 2M(M+2\mu D)\sum_{t=k+1-\tau(\gamma_{k})}^{k}\gamma_{t} +  2 M(M+2\mu D)\gamma_{k}\leq 2 M(M+2\mu D)\big[\tau(\gamma_{k})\gamma_{k-\tau(\gamma_{k})}+\gamma_{k}\big],
\label{apx_lem_sc_mixing_grad:Eq1e}
\end{align}
where the last inequality is due to the fact that $\gamma_{k}$ is nonincreasing. Finally, using \eqref{notation:DM} we consider the last term of Eq.\ \eqref{apx_lem_sc_mixing_grad:Eq1} 
\begin{align}
\frac{\mu\gamma_{k}}{1+\mu\gamma_{k}}\langle G(y_{k};\xi_{k}) - \nabla f(y_{k})\,,\, x_{k} - y_{k}\rangle \leq \frac{4\mu MD\gamma_{k}}{1+\mu\gamma_{k}}\cdot    \label{apx_lem_sc_mixing_grad:Eq1f} 
\end{align}
Taking the expectation on both sides of \eqref{apx_lem_sc_mixing_grad:Eq1} and using Eqs.\ \eqref{apx_lem_sc_mixing_grad:Eq1d}--\eqref{apx_lem_sc_mixing_grad:Eq1f} immediately gives Eq.\ \eqref{apx_lem_sc_mixing_grad:Ineq}.
\end{proof}

\begin{proof}[Proof of Theorem \ref{thm:sc}]
First, using \eqref{notation:Gamma} and \eqref{thm_sc:stepsizes} gives $\Gamma_{0} = 1$, $\alpha_{0} = 2$, and $\gamma_{0}=2/\mu$. Second, it is straightforward to verify that \eqref{thm_sc:stepsizes} satisfies  \eqref{condition:stepsizes} and  \eqref{lem_f_xbar:stepsizes}. Thus, using \eqref{apx_lem_sc_mixing_grad:Ineq} into \eqref{lem_f_xbar:Ineq} and since $L=0$ we have
\begin{align}
f(\xbar_{k}) - f(x^*) &\leq [f(\xbar_{0}) + 3\mu D]\Gamma_{k} + \Gamma_{k}\sum_{t=1}^{k} \frac{4M^2\gamma_{t}\alpha_{t}}{\Gamma_{t}(1+\mu\gamma_{t})}\notag\\
&\quad+ (2M^2 +4\mu MD+24\mu LD^2)
\Gamma_{k}\sum_{t=1}^{k}\frac{\tau(\gamma_{t})\alpha_{t}\gamma_{t-\tau(\gamma_{t})}}{\Gamma_{t}}\notag\\
&\quad + 2(D+M^2+4\mu MD)\Gamma_{k}\sum_{t=1}^{k}\frac{\alpha_{t}\gamma_{t}}{\Gamma_{t}}\cdot\label{thm_sc:Eq1}
\end{align}
Next, consider each summand on the right-hand side of \eqref{thm_sc:Eq1}. Using \eqref{thm_sc:stepsizes} and \eqref{notation:Gamma} (to have $\Gamma_{t} = 2/t(t+1)$) yields
\begin{align}
&\sum_{t=1}^{k}\frac{\gamma_{t}\alpha_{t}}{\Gamma_{t}(1+\mu\gamma_{t})} = \sum_{t=1}^{k}\frac{4 t(t+1)}{2\mu(t+1)^2(1+\frac{2}{t+1})}= \sum_{t=1}^{k}\frac{2 t}{\mu(t+3)} \leq \frac{2k}{\mu}\cdot\label{thm_sc:Eq1a}
\end{align}
Using \eqref{notation:tau} and \eqref{thm_sc:stepsizes} gives $\tau(\gamma_{k}) = \log(\mu(k+1)/2)$. Theregore, $\mu (k+1) /2 \geq \tau(\gamma_{k})$ and we obtain 
\begin{align*}
\gamma_{k-\gamma(\gamma_{k})} = \frac{2}{\mu(k+1-\log(\mu(k+1)/2))} \leq \frac{(2+\mu)\tau(\gamma_t)}{\mu k}\cdot   
\end{align*}
Using the relation above, \eqref{thm_convex:stepsizes}, and $\Gamma_{t} = 2/(t(t+1))$ gives 
\begin{align}
&\sum_{t=1}^{k}\frac{\alpha_{t}\tau(\gamma_{t})\gamma_{t-\tau(\gamma_{t})}}{\Gamma_{t}}\leq \frac{2+\mu}{\mu}\sum_{t=1}^{k}\tau(\gamma_{t}) \leq  \frac{(2+\mu)(k+1)\log(\mu(k+1)/2)}{\mu}\cdot\label{thm_sc:Eq1b}
\end{align}
Finally, we consider
\begin{align}
\sum_{t=1}^{k}\frac{\alpha_{t}\gamma_{t}}{\Gamma_{t}} \leq \frac{2k}{\mu}\cdot\label{thm_sc:Eq1c}
\end{align}
Using \eqref{thm_sc:Eq1a}--\eqref{thm_sc:Eq1c} into \eqref{thm_sc:Eq1} together with $\Gamma_{k} = 2/k(k+1)$ immediately gives \eqref{thm_sc:Ineq}, i.e.,
\begin{align*}
 f(\xbar_{k}) - f(x^*) &\leq \frac{2f(\xbar_{0}) + 6\mu D}{k(k+1)} + \frac{8M^2}{\mu(k+1)} + \frac{2(D + M^2+4\mu MD)}{\mu(k+1)}\notag\\
&\qquad+ \frac{4(M^2 +2\mu MD+12\mu LD^2)
(2+\mu)[\log(\mu/2)+\log(k+1)]}{\mu k}\notag\\
&=  \frac{2f(\xbar_{0}) + 6\mu D}{k(k+1)}+ \frac{2D + 10M^2 + 8\mu MD}{\mu(k+1)} \\ 
& \qquad + \frac{4(M^2 +2\mu MD+12\mu LD^2)
(2+\mu)\log(\frac{\mu(k+1)}{2})}{\mu k}\cdot 
\end{align*}
\end{proof}

\end{document}